\documentclass[onefignum,onetabnum]{siamart190516}
\usepackage{amsmath,amssymb,textcomp,xifthen,psfrag,graphicx,color}
\usepackage{color}

\usepackage[utf8]{inputenc}
\usepackage[T1]{fontenc}

\usepackage{braket,amsfonts,amssymb}

\usepackage{array}

\usepackage[caption=false]{subfig}
\usepackage{pgfplots}
\usepgfplotslibrary{groupplots}
\usepgfplotslibrary{colorbrewer}
\pgfplotsset{compat=newest}

\newsiamthm{claim}{Claim}
\newsiamremark{remark}{Remark}
\newsiamremark{hypothesis}{Hypothesis}
\crefname{hypothesis}{Hypothesis}{Hypotheses}

\usepackage{algorithmic}

\usepackage{graphicx,epstopdf}

\Crefname{ALC@unique}{Line}{Lines}

\usepackage{amsopn}

\newcommand{\dual}[2]{\langle#1\hspace*{.5mm},#2\rangle}
\newcommand{\vdual}[2]{(#1,#2)}

\newcommand{\norm}[3][]{#1\|#2#1\|_{#3}}

\newcommand{\DD}{\mathcal{D}}

\newcommand\divop{\operatorname{div}}
\newcommand{\diam}{\mathrm{diam}}
\newcommand{\wilde}{\widetilde}

\newcommand\stdiv{\operatorname{div}_{t,\boldsymbol{x}}}
\newcommand\stDiv{\operatorname{\mathbf{div}}_{t,\boldsymbol{x}}}
\newcommand\divsp{\divop_{\boldsymbol{x}}}
\newcommand\Deltasp{\Delta_{\boldsymbol{x}}}
\newcommand\nablasp{\nabla_{\boldsymbol{x}}}
\newcommand\stnabla{\nabla_{t,\boldsymbol{x}}}
\newcommand\stNabla{\boldsymbol{\nabla}_{t,\boldsymbol{x}}}

\newcommand{\cA}{\ensuremath{\mathcal{A}}}

\newcommand{\cT}{\ensuremath{\mathcal{T}}}

\newcommand{\cD}{\ensuremath{\mathcal{D}}}
\newcommand{\R}{\ensuremath{\mathbb{R}}}
\newcommand{\N}{\ensuremath{\mathbb{N}}}

\newcommand{\bL}{\ensuremath{\boldsymbol{L}}}
\newcommand{\bH}{\ensuremath{\boldsymbol{H}}}

\newcommand{\g}{\ensuremath{\mathbf{g}}}
\newcommand{\f}{\ensuremath{\mathbf{f}}}

\newcommand{\nn}{\ensuremath{\boldsymbol{n}}}
\newcommand{\nnsp}{\ensuremath{\boldsymbol{n}_{\boldsymbol{x}}}}

\newcommand{\OO}{\ensuremath{\mathcal{O}}}

\newcommand{\ff}{\ensuremath{\mathbf{f}}}
\newcommand{\xx}{\ensuremath{\boldsymbol{x}}}
\newcommand{\xX}{\ensuremath{\mathcal{X}}}

\newcommand{\pphi}{{\boldsymbol\phi}}

\newcommand{\ssigma}{{\boldsymbol\sigma}}
\newcommand{\ttau}{{\boldsymbol\tau}}
\newcommand{\cchi}{{\boldsymbol\chi}}

\newcommand{\xxi}{{\boldsymbol\xi}}

\title{Well-posedness of first-order acoustic wave equations and space-time finite element
approximation\thanks{\funding{This work was supported by ANID Chile through FONDECYT projects and 1210391 (TF), and 1210579 (RG and MK)}}}

\author{Thomas F\"uhrer\thanks{Facultad de Matem\'{a}ticas, Pontificia Universidad Cat\'{o}lica de Chile, Santiago, Chile
    (\email{tofuhrer@mat.uc.cl}).}
\and Roberto Gonz\'alez\thanks{Departamento de Matem\'atica, Universidad T\'ecnica Federico Santa Mar\'ia, Valpara\'iso, Chile
(\email{roberto.gonzalezf@sansano.usm.cl})}
\and Michael Karkulik\thanks{Departamento de Matem\'atica, Universidad T\'ecnica Federico Santa Mar\'ia, Valpara\'iso, Chile
(\email{michael.karkulik@usm.cl})}}
\begin{document}
\maketitle
\begin{abstract}
  We study a first-order system formulation of the (acoustic) wave equation and prove that the operator of this system is an isomorphsim from an appropriately defined graph space to $L^2$.
  The results rely on well-posedness and stability of the weak and ultraweak formulation of the second-order wave equation.
  As an application we define and analyze a space-time least-squares finite element method for solving the wave equation.
  Some numerical examples for one- and two-dimensional spatial domains are presented.
\end{abstract}


\section{Introduction}
For a finite time interval $J=(0,T)$, a Lipschitz domain $\Omega\subset\R^d$, and functions
$v:J\times\Omega\rightarrow\R$ and $\ssigma:J\times\Omega\rightarrow\R^d$,
our object of study is the following initial value problem for a first-order hyperbolic partial differential equation
\begin{subequations}\label{intro:eq}
\begin{alignat}{2}
  \partial_t v-\divsp\ssigma &= f &\quad&\text{in }J\times \Omega,\\
  \partial_t\ssigma-\nablasp v &= \g &\quad&\text{in }J\times \Omega, \\
    v(0,\cdot) &= v_0 &\quad&\text{in } \Omega, \\
    \ssigma(0,\cdot) &= \ssigma_0&\quad&\text{in } \Omega,
\end{alignat}
\end{subequations}
where $f$, $\g$, $v_0$, $\ssigma_0$ are given data and homogeneous Dirichlet boundary conditions on $J\times\partial\Omega$ are imposed on $v$.
Hyperbolic partial differential equations arise in a variety of disciplines where wave or advection fenomena are to be modeled, such as acoustics,
elastodynamics, and electromagnetism. In fact, the solution of the second order acoustic wave equation
\begin{align}\label{eq:model:wave}
  \begin{split}
    \partial_{tt}u - \Deltasp u &= f \text{ in }J\times \Omega, \quad u = 0 \text{ on } J\times\partial\Omega,\\
    u(0,\cdot) &= u_0, \quad    \partial_t u(0,\cdot) = u_1 \text{ in } \Omega
  \end{split}
\end{align}
defines, formally, a solution of~\eqref{intro:eq} by setting $v = \partial_t u$, $\ssigma = \nablasp u$ and $\g=0,v_0=u_1,\ssigma_0=\nablasp u_0$.
Naturally, the numerical approximation of hyperbolic partial differential equations
is required in a wide range of applications. A variety of methods is available,
ranging from pure finite differences to time-stepping with finite elements and discontinuous Galerkin methods,
as well as finite volume methods, cf. the textbooks~\cite{CohenP_17,Cohen_02,HesthavenW_08,Leveque_02}.
Various other approaches do not exactly fit into these frameworks, and may be called \textit{space-time methods}. Those may be separated into methods using
finite elements in space and time but not simultaneously and hence requiring tensor-product meshes,
cf. eg.,~\cite{BalesLasiecka_94,BangerthGR_10,FrenchP_96,HulbertH_90,Johnson_93,MatuteDR_22,SteinbachZ_20},
and methods using finite elements in time and space simultaneously,
cf. eg.,~\cite{DoerflerFW_16,GopalakrishnanSW_17,GopalakrishnanS_19,MoiolaP_18,PerugiaSSW_20}.
The main difficulty for the development of simultaneous space-time discretizations is that
standard variational formulations of evolution equations often involve different trial and test spaces
so that pairs of discretization spaces have to be chosen cautiously in order to obtain a robust
Petrov--Galerkin method.
In order to circumvent this problem, we mimic an approach we used successfully for the heat equation, cf.~\cite{FuehrerK_21}.
Using the operators
\begin{align*}
  \cA(v,\ssigma) = 
  \begin{pmatrix}
    \cA_0(v,\ssigma)\\
    v(0,\cdot)\\
    \ssigma(0,\cdot)
  \end{pmatrix}, \qquad\text{ with }\qquad
  \cA_0(v,\ssigma) =
    \begin{pmatrix}
    \partial_t v-\divsp\ssigma\\
    \partial_t\ssigma-\nablasp v
  \end{pmatrix},
\end{align*}
we can write equation~\eqref{intro:eq} as $\cA(v,\ssigma) = \f$.
In order to obtain well-posedness of this equation
for any $\ff = (f,\g,v_0,\ssigma_0)^\top\in L^2(J\times\Omega)\times L^2(J\times\Omega)^d\times L^2(\Omega)\times L^2(\Omega)^d$,
we will identify a Hilbert space $V_0$ such that
\begin{align*}
\cA:V_0 \rightarrow
  Y := L^2(J\times\Omega)\times L^2(J\times\Omega)^d \times L^2(\Omega)\times L^2(\Omega)^d
\end{align*}
is an isomporhism. This is our first main result, Thm.~\ref{thm:main}.
An immediate consequence of this mapping property is that the
variational formulation of least-squares type
\begin{align*}
  \text{find } (v,\ssigma)\in V_0 \text{ such that }\vdual{\cA(v,\ssigma)}{\cA(w,\ttau)}_{Y} = \vdual{\ff}{\cA(w,\ttau)}_{Y} \text{ for all } (w,\ttau)\in V_0
\end{align*}
is well-posed due to the Lax--Milgram lemma. Furthermore, any finite dimensional subspace
$V_{0,h}\subset V_0$ (even on locally refined space-time meshes) allows for a quasi-optimal approximate solution via the resolution
of a symmetric positive definite linear system.
\subsection{Existing works}
The related works~\cite{ErnestiW_19_A,GopalakrishnanS_19} treat the case
of homogeneous initial and Dirichlet boundary conditions by considering an adequate domain
$V_{\cA_0}$ including homogeneous initial and boundary conditions,
and study
\begin{align*}
  \cA_0:V_{\cA_0} \rightarrow Y_{\cA_0} := L^2(J\times\Omega)\times L^2(J\times\Omega)^d.
\end{align*}
In the mentioned works it is shown that $\cA_0$ and its adjoint $\cA_0^*$ are bounded from below
for smooth functions $(v,\ssigma)$ fulfilling the homogeneous initial and boundary conditions.
Assuming the density of such smooth functions in $V_{\cA_0}$ (same for $\cA_0^*$), standard arguments from functional analysis imply that
$\cA_0:V_{\cA_0}\rightarrow Y_{\cA_0}$ is an isomorphism.
However, the required density results are non-trivial
and are shown in~\cite{GopalakrishnanS_19} only for rectangular domains $\Omega$.
Boundedness from below for first-order reformulations of second-order partial differential equations
has been employed to analyze least-squares finite element methods for elliptic equations already
in~\cite{CaiLMM_94,CaiMM_97}.
Such an approach requires adequate stability results of the original second-order equation with right-hand side in negative norms $H^{-1}$ in order
to get rid of second order derivatives.
Second-order parabolic equations allow for right-hand sides
in $L^2(J;H^{-1})$, cf.~\cite[Thm.~23.A]{Zeidler_90}, and we used this fact in our previous
work~\cite{FuehrerK_21} in the case of the heat equation, cf. also~\cite{GantnerS_21}.
Unfortunately, second-order hyperbolic equations such as~\eqref{eq:model:wave} do not, at least not in their ubiquitous standard variational form,
cf.~\cite[Thm.~4.2.24]{Zank_19}.

We circumvent this lack of stability in $H^{-1}$ by employing an
ultraweak formulation of the second-order wave equation from~\cite[Ch.~3, Sec.~9]{LionsM_72},
which does enjoy the required stability property and allows us to conclude that the operator $\cA$ is indeed
an isomorphism as stated above. Ultraweak formulations of wave equations have been
previously used for purposes in numerical analysis, cf.~\cite{HenningPSU_22}.

\subsection{Overview}
The remainder of this work is organized as follows: In Section~\ref{sec:preliminaries} we introduce notations, definitions, state and collect auxiliary results from the literature. Section~\ref{sec:main} contains the statements of our main results followed by their proofs. Particularly, we introduce space $V_0$, Theorem~\ref{thm:main} states that operator $\cA$ is an isomorphism and in Theorem~\ref{thm:density} we show that the space of smooth functions satisfying homogeneous boundary conditions is dense in $V_0$. 
As application of our main results we cover a least-squares finite element method together with an adaptive algorithm in Section~\ref{sec:LS}. Finally, some numerical experiments are presented in Section~\ref{sec:num}.

\section{Preliminaries}\label{sec:preliminaries}

Throughout, $\Omega\subset\R^d$ denotes a spatial polygonal Lipschitz domain and $J=(0,T)$ a bounded time interval.
We will also use the space-time cylinder $Q=J\times \Omega$.
\subsection{Function spaces}
For $\omega$ a computational domain (which can be $J,\Omega,Q$ or their closed versions
$\overline J, \overline\Omega,\overline Q$), and a Banach space
$X$, we use various spaces of $X$-valued functions: continuous functions $C(\omega;X)$,
smooth functions $C^\infty(\omega;X)$, and compactly supported smooth functions
$\cD(\omega;X)$. 
The space of distributions is denoted by $\cD'(\omega;X)$.
In the case $X = \R$ we omit $X$ in the preceding (and following)
notations.
We consider the standard Lebesgue and Sobolev spaces
$L^2(\Omega)$ and $H^k(\Omega)$ for $k\geq 1$ with the standard norms.
Vector-valued versions in $d$ dimensions of these spaces will be denoted by bold-face symbols, e.g.,
$\bL^2(\Omega) := L^2(\Omega)^d$.
The spatial gradient will be denoted by $\nablasp$.
The space $H^1_0(\Omega)$ consists of all functions $u\in H^1(\Omega)$
with vanishing trace $\gamma u$ on the boundary $\partial\Omega$.
On $H^1_0(\Omega)$, we use the equivalent norm
\begin{align*}
  \| u \|_{H^1_0(\Omega)} := \| \nablasp u \|_{L^2(\Omega)}.
\end{align*}
Note that $\nablasp H^1_0(\Omega) := \left\{ \nablasp u \mid u\in H^1_0(\Omega) \right\}$
is a closed subspace of $\bL^2(\Omega)$, and we have the direct sum, i.e., Helmholtz decomposition,
$\bL^2(\Omega) = \nablasp H^1_0(\Omega) \oplus \nablasp H^1_0(\Omega)^\perp$. We will use the
$L^2$-orthogonal projections $\Pi_{\nablasp H^1_0(\Omega)}$ and $\Pi_{\nablasp H^1_0(\Omega)^\perp}$
onto these subspaces.

We define $H^{-1}(\Omega):= H^1_0(\Omega)'$ as topological dual
with respect to the extended $L^2(\Omega)$ scalar product $\vdual{\cdot}{\cdot}_\Omega$.
A standard reference on Lebesgue--Sobolev spaces is~\cite{AdamsF_03}.
We will use the Lebesgue--Bochner space $L^2(X)$ of functions
$f:J\rightarrow X$ which are strongly measurable with respect to the Lebesgue measure $ds$ on $J$ and
\begin{align*}
  \| f \|_{L^2(X)}^2 := \int_J \| f(s) \|_X^2\,ds < \infty.
\end{align*}
For a function $f\in L^2(X)$ we define $f'\in \DD'(J;X)$ by
\begin{align*}
  f'(\varphi) = -\int_J f(s)\cdot \varphi'(s)\,ds\in X\quad\forall\varphi \in \cD(J).
\end{align*}
We then define the Sobolev--Bochner space $H^k(X)$ of functions in $L^2(X)$ whose weak derivatives
$f^{(\alpha)}$ of all orders $|\alpha|\leq k$ belong to $L^2(X)$, endowed with the norm
\begin{align*}
  \| f \|_{H^k(X)}^2 := \sum_{|\alpha|\leq k} \| f^{(\alpha)} \|_{L^2(X)}^2.
\end{align*}
An extensive reference on Sobolev--Bochner spaces is~\cite{HytonenNVW_16}.
We have the following well-known results, cf.~\cite[Ch.~XVIII]{DautrayL_92}.
\begin{lemma}\label{lem:ttrace}
  \begin{enumerate}
    \item[(i)] The embedding $H^1(X)\hookrightarrow C(\overline J;X)$ is continuous, and
      \begin{align*}
	f(t) - f(s) = \int_s^t f'(r)\,dr \qquad\text{ for all } f \in H^1(X) \text{ and } s,t\in\overline J.
      \end{align*}
    \item[(ii)]
      Let $X\hookrightarrow H \hookrightarrow X'$ be a Gelfand triple.
      Then, the embedding $L^2(X)\cap H^1(X') \hookrightarrow C(\overline J;H)$
      is continuous.
    \item[(iii)] For $f\in L^2(X)$ and $\varphi\in X'$ there holds $\varphi(\int_J f\,ds) = \int_J\varphi(f)\,ds$.
    \item[(iv)] For $X\hookrightarrow H$ continuous and dense, the space $C^\infty(\overline J;X)$ is dense in
      $L^2(X)\cap H^1(H)$.
    \item[(v)] Space $\DD(J)\otimes X$ is dense in $L^2(X)$.
  \end{enumerate}
\end{lemma}
For a space of real-valued functions $U$ and a Banach space $X$, we will denote 
finite tensor products as
\begin{align*}
  U\otimes X := \Bigl\{ \sum_{j=1}^n \varphi_j v_j \mid n\in\N, \varphi_j\in U,v_j\in X \Bigr\}.
\end{align*}%
\begin{lemma}\label{lem:density}
      For $X\hookrightarrow H$ continuous and dense and $\wilde X$ dense in $X$, the space
      $C^\infty(\overline J)\otimes \wilde X$ is dense in $L^2(X)\cap H^1(H)$.
\end{lemma}
\begin{proof}
  By Lemma~\ref{lem:ttrace}~(iv), the space $C^\infty(\overline J;X)$ is dense in 
    $L^2(X)\cap H^1(H)$, thus, $H^1(X)$ is dense in $L^2(X)\cap H^1(H)$.
    It is sufficient to prove that $C^\infty(\overline J)\otimes \wilde X$ is dense in $H^1(X)$.
    Let $u\in H^1(X)$ be given and let $\varepsilon>0$.
    Since $\DD(J)\otimes X$ is dense in $L^2(X)$ (Lemma~\ref{lem:ttrace} (v)) there exists $n\in \N$, $\varphi_j\in \DD(J)$, $v_j\in X$, $j=1,\dots,n$ such that 
    \begin{align*}
      \|u'-\psi_\varepsilon\|_{L^2(X)} < \varepsilon, \quad \psi_\varepsilon = \sum_{j=1}^n \varphi_jv_j.
    \end{align*}
    Note that $u(t) = u(0)+\int_0^t u'(s)\,ds$ and define 
    \begin{align*}
      u_\varepsilon(t) = u(0) + \sum_{j=1}^n \int_0^t \varphi_j(s)\,ds \, v_j.
    \end{align*}
    Set $\phi_j(t) = \int_0^t\varphi_j(s)\,ds$, $\phi_{n+1}(t) = 1$, $v_{n+1} = u(0)$. Then, $u_\varepsilon = \sum_{j=1}^{n+1} \phi_jv_j \in C^\infty(\overline J)\otimes X$. 
    Furthermore, $u_\varepsilon' = \psi_\varepsilon$ and $u_\varepsilon(0) = u(0)$.
    Together with an application of a H\"older estimate in the time variable we obtain
    \begin{align*}
      \|u-u_\varepsilon\|_{L^2(X)} + \|u'-u_\varepsilon'\|_{L^2(X)} \leq C \varepsilon, 
    \end{align*}
    where $C>0$ is independent of $\varepsilon$. We conclude that $C^\infty(\overline J)\otimes X$ is dense in $H^1(X)$. Since $\wilde X$ is dense in $X$ we can replace $v_j\in X$ in the above proof by sufficiently close $v_j\in \wilde X$.
\end{proof}
Note that $L^2(L^2(\Omega))$ is isomorphic to $L^2(Q)$. From this point on we identify both spaces. 
Let $\partial_t f\in D'(Q)$ denote the distributional derivative.
For our studies we need to work with time derivatives and space derivatives simultaneously. The latter are defined as distributions on $Q$. The next result shows that $f'$ can be identified with the distributional derivative $\partial_t f\in \DD'(Q)$. Its proof is straightforward and omitted. 
\begin{lemma}
  Mapping $\{f'\mid f\in L^2(Q)\} \to \DD'(Q)$, $
  f' \mapsto \partial_t f$ is inyective. 
\end{lemma}
For a linear and bounded operator $L:X\rightarrow Y$ and $\varphi\in \cD(J)\otimes X$ we can define $L\varphi \in C(\overline J;Y)$ pointwise by $(L\varphi)(t) := L(\varphi(t))$. Then, 
\begin{align*}
  \| L\varphi \|_{L^2(Y)} \leq \|L\|_{X\rightarrow Y}\| \varphi \|_{L^2(X)},
\end{align*}
and we can define $Lf\in L^2(Y)$ for $f\in L^2(X)$ by density (Lemma~\ref{lem:ttrace} (v)).
We will have to interchange spatial derivatives with temporal derivatives and evaluations.
To that end, we will employ the following lemma.
\begin{lemma}\label{lem:schwarz}
  Let $L:X\rightarrow Y$ be a linear and bounded operator between two Banach spaces $X$ and $Y$
  and $u\in H^1(X)$. Then,
  \begin{enumerate}
    \item[(i)] $(Lu)(t) = L(u(t))$ for all $t\in \overline J$,
    \item[(ii)] $Lu \in H^1(Y)$ and $(Lu)' = L(u')$.
  \end{enumerate}
\end{lemma}
\begin{proof}
  The first statement follows by definition and Lemma~\ref{lem:ttrace}, (ii).
  The second statement follows from~\cite[Lemma~64.34]{ErnG_21PartIII}.
\end{proof}
Note that $\partial_{\xx_j}:L^2(L^2(\Omega))\rightarrow L^2(H^{-1}(\Omega))$ is bounded. Likewise, the spatial gradient
$\nablasp:L^2(L^2(\Omega))\rightarrow L^2(\bH^{-1}(\Omega))$ is bounded.
By $\divsp:\bL^2(\Omega)\rightarrow H^{-1}(\Omega)$ respectively
$\divsp:L^2(\bL^2(\Omega))\rightarrow L^2(H^{-1}(\Omega))$
we denote the spatial divergence operator.
Set $\bL^2(Q) = L^2(Q;\R^d)$.

Finally, for $\ttau = (\ttau_1,\dots,\ttau_{1+d})\in L^2(Q)^{1+d}$ we define the time-space divergence operator by
$\stdiv\ttau = \partial_t\ttau_1 +  \sum_{j=1}^{d} \partial_{x_j} \ttau_{j+1}$.
We have the following observation, see, e.g.,~\cite[Theorem~4.5]{BerggrenH21}.
\begin{lemma}\label{lem:subspace}
  Let $X,Y,Z$ be Banach spaces with continuous embedding $Z\hookrightarrow Y$. Let $T:X\rightarrow Y$
  be linear and continuous. Then, $W := \left\{ x \in X \mid Tx\in Z \right\}$
  is a Banach space with respect to the (squared) norm
  \begin{align*}
    \| w \|_{W}^2 := \| w \|_X^2 + \| Tw\|_Z^2.
\end{align*}
\end{lemma}
\subsection{Variational formulations of the second order wave equation}
In this section we collect results on different variational formulations of the wave equation found in the literature. We also need an additional regularity result for the weak formulation of the wave equation, see Lemma~\ref{lem:regularity} below. We start with the following well-posedness and stability result of the weak formulation of the wave equation~\eqref{eq:model:wave}.
\begin{theorem}[{\cite[Ch.~3, Thm.~8.1]{LionsM_72}}]\label{thm:wave:weak}
Let $f\in L^2(L^2(\Omega))$, $g_0\in H^1_0(\Omega)$, $g_1\in L^2(\Omega)$ be given. There exists a unique
$u\in L^2(H^1_0(\Omega))\cap H^1(L^2(\Omega))\cap H^2(H^{-1}(\Omega))$ such that for all $v\in L^2(H_0^1(\Omega))$
\begin{align}\label{eq:wave}
  \begin{split}
    \int_0^T\dual{\partial_{tt} u(s) - \Deltasp u(s)}{v(s)}_{H^{-1}(\Omega)\times H^1_0(\Omega)}\,ds &= \int_0^T \vdual{f(s)}{v(s)}_\Omega\,ds\\
    u(0) &= g_0, \quad \partial_t u(0) = g_1.
  \end{split}
\end{align}
Furthermore, $u\in C(\overline J;H^1_0(\Omega))\cap C^1(\overline J;L_2(\Omega))$ and
\begin{align}\label{eq:wave:weak}
\begin{split}
  &\| u \|_{L^\infty(H^1_0(\Omega))} + \| \partial_t u \|_{L^\infty(L^2(\Omega))} + \| \partial_{tt}u \|_{L^2(H^{-1}(\Omega))}
  \\
  &\qquad \lesssim \| f \|_{L^2(Q)} + \| g_0 \|_{H^1(\Omega)} + \| g_1 \|_{L^2(\Omega)}.
\end{split}
\end{align}
\end{theorem}
Using the change of variables $t\mapsto T-t$, we may also impose terminal conditions at the
final time $t=T$ instead of initial conditions at $t=0$.
The next regularity result is obtained by following the arguments from~\cite[Section~7.2]{Evans}.
We stress that we do not require $\Omega$ to be smooth compared to~\cite[Section~7.2]{Evans}
(which is needed to show improved regularity in the space variable).
Further, we state the result with vanishing initial conditions, although it
is of course valid with vanishing terminal conditions.
\begin{lemma}\label{lem:regularity}
  Suppose that $u\in L^2(H^1_0(\Omega))\cap H^1(L^2(\Omega))\cap H^2(H^{-1}(\Omega))$ is
  the unique weak solution of the wave equation~\eqref{eq:wave} with $f\in H^1(L^2(\Omega))$ and
  initial conditions $u(0)=0$, $\partial_t u(0)=0$. Then, besides~\eqref{eq:wave:weak} with $g_0=0$, $g_1=0$ we have that $\partial_t u\in L^\infty(H_0^1(\Omega))$, $\partial_{tt}u\in L^\infty(L^2(\Omega))$ with
  \begin{align*}
    \| \nablasp \partial_t u\|_{L^\infty(\bL^2(\Omega))} + \| \partial_{tt} u\|_{L^\infty(L^2(\Omega))}\lesssim \norm{\partial_t f}{L^2(Q)}.
  \end{align*}
\end{lemma}
\begin{proof}
  Let $u_m\colon [0,T]\to H_0^1(\Omega)$ be given by
  \begin{align*}
    u_m(t) = \sum_{j=1}^m d_j^m(t)w_j
  \end{align*}
  where $(w_j)_{j\in\N}$ denotes a complete orthonormal system in $L^2(\Omega)$, i.e., $\vdual{w_j}{w_k}_\Omega = \delta_{jk}$ and orthogonal system in $H_0^1(\Omega)$, i.e., $\vdual{\nablasp w_j}{\nablasp w_k}_\Omega = 0$ for $j\neq k$. 
  Coefficients $d_j^m$ are uniquely determined through the wave equation
  \begin{align}\label{eq:regularity:proof1}
    \vdual{\partial_{tt}u_m(t)}v_\Omega + \vdual{\nablasp u_m(t)}{\nablasp v}_\Omega = \vdual{f(t)}v_\Omega \quad\forall v\in H_0^1(\Omega)
  \end{align}
  with $d_j^m(0) = 0$, $(d_j^m)'(0) = 0$, see~\cite[Section~7.2]{Evans}.
  Particularly, in~\cite[Section~7.2.2]{Evans} it is shown that there exists a subsequence which weakly converges to the variational solution $u$ of~\eqref{eq:wave}.
  Following the arguments in the proof of~\cite[Section~7.2, Theorem~5]{Evans} we find that
  \begin{align*}
    \|\partial_{tt}u_m(t)\|_\Omega^2 + \|\nablasp \partial_t u_m(t)\|_\Omega^2 \lesssim \|\partial_t f\|_Q^2 + \|\nablasp \partial_t u_m(0)\|_\Omega^2 + \|\partial_{tt}u_m(0)\|_\Omega^2.
  \end{align*}
  It only remains to estimate the terms $\|\nablasp \partial_t u_m(0)\|_\Omega$ and $\|\partial_{tt}u_m(0)\|_\Omega$. For the former note that $(d_j^m)'(0) = 0$, therefore, $\partial_t u_m(0) = 0$. For the latter, we evaluate~\eqref{eq:regularity:proof1} at $t=0$ and get that
  \begin{align*}
    \vdual{\partial_{tt}u_m(0)}v_\Omega = \vdual{f(0)}{v}_\Omega - \vdual{\nablasp u_m(0)}{\nablasp v}_\Omega \quad\forall v\in H_0^1(\Omega).
  \end{align*}
  Since $u_m(0) = 0$ and $H_0^1(\Omega)$ is dense in $L^2(\Omega)$ we conclude that $\|\partial_{tt}u_m(0)\|_\Omega = \|f(0)\|_\Omega$. 
  Putting all estimates together and passing to the limit $m\to\infty$ proves the asserted statement.
\end{proof}
Following~\cite[Ch.~3, Sec.~9]{LionsM_72}, we may use Theorem~\ref{thm:wave:weak}
to define the linear space
\begin{align*}
  \xX = \left\{ \varphi = (\partial_{tt} - \Deltasp)^{-1}g \mid g \in L^2(L^2(\Omega)), \varphi(T)=\partial_t\varphi(T)=0 \right\}.
\end{align*}
Next,~\cite[Ch.~3, Sec.~9, Thm.~9.3, Thm.~9.4, Remarks~9.5 and~9.11]{LionsM_72} state the following.
\begin{theorem}[Ultraweak wave equation]\label{lm:wave:uw}
  Let $f\in L^2(H^{-1}(\Omega))$, $g_0\in L^2(\Omega)$, and $g_1\in H^{-1}(\Omega)$. There is a unique function
  $u\in L^2(L^2(\Omega))\cap H^1(H^{-1}(\Omega))$ such that
  \begin{align*}
    &\int_0^T \vdual{u(s)}{\partial_{tt}\varphi(s)-\Deltasp \varphi(s)}_\Omega\,ds \\
    &\qquad= \int_0^T \dual{f(s)}{\varphi(s)}_{H^{-1}(\Omega)\times H^1_0(\Omega)}\,ds + \dual{g_1}{\varphi(0)}_{H^{-1}(\Omega)\times H^1_0(\Omega)} -\vdual{g_0}{\partial_t\varphi(0)}_{\Omega},
  \end{align*}
  for all $\varphi\in\xX$ with $\Deltasp \varphi\in L^2(L^2(\Omega))$. There holds
  \begin{align*}
    \| u \|_{L^2(L^2(\Omega))} + \| \partial_t u \|_{L^2(H^{-1}(\Omega))} \lesssim \| f \|_{L^2(H^{-1}(\Omega))}
    + \| g_0 \|_{L^2(\Omega)} + \| g_1 \|_{H^{-1}(\Omega)}.
  \end{align*}
  In addition, $u\in C(\overline J;L^2(\Omega))\cap C^1(\overline J;H^{-1}(\Omega))$ with
  $u(0) = g_0$, $\partial_t u(0)=g_1$.
\end{theorem}
\section{First order acoustic wave equations}\label{sec:main}
In a first step, we define the following algebraic subspace of $L^2(Q)^{1+d}$,
\begin{align*}
  \wilde V:&=\left\{ (v,\ssigma)\in L^2(Q)^{1+d} \mid \cA_0(v,\ssigma)\in L^2(Q)^{1+d}\right\}
  \\
  &= \{ (v,\ssigma) \in L^2(Q)^{1+d} \mid \stdiv(v,-\ssigma)\in L^2(Q), \, \partial_t\ssigma- \nablasp v\in \bL^2(Q) \}.
\end{align*} 
Let $\|\cdot\|_{\wilde V}$ denote the graph norm corresponding to $\wilde V$, i.e.,  for $(v,\ssigma)\in \wilde V$,
\begin{align*}
  \|(v,\ssigma)\|_{\wilde V}^2 := \|(v,\ssigma)\|_{L^2(Q)\times \bL^2(Q)}^2 +
  \|\stdiv(v,-\ssigma)\|_{L^2(Q)}^2 + \|\partial_t\ssigma-\nablasp v\|_{\bL^2(Q)}^2. 
\end{align*}
The next result is immediate by virtue of closedness of differential operators, cf.~\cite{Jensen_PAMM_06}.
\begin{lemma}
  The space $(\wilde V,\|\cdot\|_{\wilde V})$ is a Hilbert space.
\end{lemma}
In order to include initial conditions we need access to traces. The following result shows existence of a well-defined trace operator on $\wilde V$. 
\begin{lemma}\label{lem:trace}
  The estimate
  \begin{align*}
    \|\partial_t v\|_{L^2(H^{-1}(\Omega))} + \|\partial_t \ssigma\|_{L^2(\bH^{-1}(\Omega))}
    \lesssim \|(v,\ssigma)\|_{\wilde V}
  \end{align*}
  holds for all $(v,\ssigma)\in \wilde V$. In particular, the trace operator
  \begin{align*}
    \gamma_0\colon \begin{cases}
      \wilde V &\to H^{-1}(\Omega)\times \bH^{-1}(\Omega), \\
    (v,\ssigma)&\mapsto (v(0),\ssigma(0))
  \end{cases}
  \end{align*}
  is a well-defined linear and bounded operator.
\end{lemma}
\begin{proof}
  We follow similar arguments as in~\cite[Lemma~2]{AntonicBV_13} resp.~\cite[Lemma~2.1]{GantnerS_21}. 
  Let $(v,\ssigma)\in \wilde V$ be given. 
  First, note that $\divsp\ssigma \in L^2(H^{-1}(\Omega))$ with 
  \begin{align*}
    \|\divsp\ssigma\|_{L^2(H^{-1}(\Omega))} \lesssim \|\ssigma\|_{\bL^2(Q)}.
  \end{align*}
  Since $\stdiv(v,-\ssigma) \in L^2(Q)$ we get that $\partial_t v \in L^2(H^{-1}(\Omega))$ with
  \begin{align*}
    \|\partial_t v\|_{L^2(H^{-1}(\Omega))} &\leq \|\stdiv(v,-\ssigma)\|_{L^2(H^{-1}(\Omega))} + \|\divsp\ssigma\|_{L^2(H^{-1}(\Omega))}
    \\
    &\lesssim \|\stdiv(v,-\ssigma)\|_{L^2(Q)} + \|\ssigma\|_{\bL^2(Q)} \lesssim \|(v,\ssigma)\|_{\wilde V}.
  \end{align*}
  Lemma~\ref{lem:ttrace} shows that $v\in C^0(\overline J;H^{-1}(\Omega))$, thus, $v(0)\in H^{-1}(\Omega)$ with 
  \begin{align*}
    \|v(0)\|_{H^{-1}(\Omega)} \lesssim \|v\|_{L^2(H^{-1}(\Omega))} + \|\partial_t v\|_{L^2(H^{-1}(\Omega))} \lesssim \|(v,\ssigma)\|_{\wilde V}
  \end{align*}
  It remains to prove analogous estimates for $\partial_t\ssigma$. 
  They follow with similar arguments as before. Therefore, we omit further details.
\end{proof}

In order to allow for the incorporation of initial conditions, we define in a second step
the space
\begin{align*}
  V := \left\{ (v,\ssigma)\in \wilde V \mid \gamma_0(v,\ssigma)\in L^2(\Omega)\times \bL^2(\Omega) \right\}
\end{align*}
with graph norm $\|(v,\ssigma)\|_{V}^2 = \|(v,\ssigma)\|_{\wilde V}^2 + \|(v(0),\ssigma(0))\|_{L^2(\Omega)\times \bL^2(\Omega)}^2$.
\begin{lemma}
  The space $(V,\|\cdot\|_{V})$ is a Hilbert space.
\end{lemma}
\begin{proof}
  The result is an application of Lemma~\ref{lem:subspace} with $X = \wilde V$, $Z=L^2(\Omega)\times \bL^2(\Omega)$, $Y=H^{-1}(\Omega)\times \bH^{-1}(\Omega)$, and $T=\gamma_0$.
\end{proof}

In a third step, we need to incorporate Dirichlet boundary conditions.
Given that elements of $V$ have reduced regularity, this is not possible via trace operators.
We shall thus restrict to a weak enforcement of homogeneous Dirichlet
boundary conditions in $V$. To that end consider the space 
\begin{align*}
  Y = \left\{\varphi \in H^1(H_0^1(\Omega))\cap H^2(L^2(\Omega)) \cap L^2(H(\Deltasp;\Omega))\mid \varphi(T) = 0,\, \partial_t\varphi(T) = 0\right\} 
\end{align*}
where $L^2(H(\Deltasp;\Omega)) = \{v\in L^2(Q)\mid \Deltasp v \in L^2(Q)\}$, equipped with the norm
\begin{align*}
  \|\varphi\|_Y^2 := \|\nablasp \varphi\|_Q^2 + \|\nablasp \partial_t\varphi\|_Q^2 + \|\Deltasp \varphi\|_Q^2 + \|\partial_{tt}\varphi\|_Q^2.
\end{align*}
Define the linear operator $D\colon V\to Y'$ for $(v,\ssigma)\in V$, $\varphi\in Y$ by
\begin{align*}
  D(v,\ssigma)(\varphi)=  \vdual{\partial_t\ssigma-\nablasp v}{\nablasp\varphi}_Q + \vdual{\ssigma}{\nablasp\partial_t\varphi}_Q
    - \vdual{v}{\Deltasp \varphi}_Q + \vdual{\ssigma(0)}{\nablasp\varphi(0)}_{\Omega}.
\end{align*}
Note that $D$ is continuous, i.e., $\| D(v,\ssigma) \|_{Y'} \lesssim \| (v,\ssigma) \|_{V}$ for all $(v,\ssigma)\in V$, 
so that
\begin{align*}
V_0 := \ker D
\end{align*}
is a closed subspace of $V$, and therefore in particular a Hilbert space.

Operator $D$ is in fact encoding homogeneous Dirichlet boundary conditions as can be seen from the following result.
\begin{lemma}\label{lem:bcdirichlet}
  Let $(v,\ssigma)\in \wilde V$ with $v\in L^2(H^1(\Omega))$. Then, $(v,\ssigma)\in \ker D$ if and only if $v\in L^2(H_0^1(\Omega))$.
\end{lemma}
\begin{proof}
  Note that $(v,\ssigma)\in \wilde V$ and $v\in L^2(H^1(\Omega))$ imply that $\partial_t\ssigma\in \bL^2(Q)$, thus, $\ssigma \in H^1(\bL^2(\Omega))$.
  Integration by parts with $\nablasp\varphi(T) = 0$ yields
  \begin{align*}
    D(v,\ssigma)(\varphi) 
    &= \int_0^T \vdual{-\nablasp v}{\nablasp\varphi}_\Omega - \vdual{v}{\Deltasp\varphi}_\Omega \,ds = \int_0^T -\dual{v}{\nablasp\varphi \cdot\nnsp}_{\partial \Omega}\,ds.
  \end{align*}
  Here, $\nablasp\varphi\cdot\nnsp$ denotes the normal derivative of $\varphi$ on the boundary $J\times \partial\Omega$.
  Suppose that $v\in L^2(H_0^1(\Omega))$. Then, $D(v,\ssigma)(\varphi) = 0$ for all $\varphi\in Y$
  by the last identity.
  Conversely, suppose that $D(v,\ssigma)(\varphi) = 0$ for all $\varphi\in Y$ which means that
  \begin{align*}
    \int_0^T \dual{v}{\nablasp\varphi\cdot\nnsp}_{\partial\Omega} \,ds = 0 \quad\forall \varphi\in Y.
  \end{align*}
  Since $\Omega$ is a polygonal domain there exist open and pairwise disjoint smooth manifolds $\Gamma_1,\dots,\Gamma_n \subseteq \partial \Omega$ with $\bigcup_{j=1}^n \overline \Gamma_j = \partial\Omega$. Fix $j$ and let $\phi\in \DD(\Gamma_j)$. We find $\psi\in H^2(\Omega)\cap H_0^1(\Omega)$ with $\nablasp\psi\cdot\nnsp = \phi$ on $\Gamma_j$ and $\nablasp\psi\cdot\nnsp = 0$ on $\Gamma_k$ with $k\neq j$. For $\mu\in \DD(J)$ we have that, using $\varphi = \mu\psi \in Y$ so that $\nablasp\varphi\cdot\nnsp = \mu\phi$,
  \begin{align*}
    \int_0^T \dual{v}{\mu\phi}_{\Gamma_j} \,ds = 0 \quad\forall \mu\in \DD(J), \, \phi\in \DD(\Gamma_j).
  \end{align*}
  By density of $\DD(J)\otimes \DD(\Gamma_j)$ in $L^2(L^2(\Gamma_j))$ we conclude that $v|_{J\times\Gamma_j} = 0$.
  This clearly holds for all $j=1,\dots,n$ and therefore finishes the proof.
\end{proof}
\subsection{Main results}
In this section we collect our main results. The first one establishes that the first-order acoustic wave
operator is an isomorphism. 
\begin{theorem}\label{thm:main}
  Operator $\cA\colon V_0 \to L^2(Q)\times \bL^2(Q) \times L^2(\Omega) \times \bL^2(\Omega)$ is an isomorphism.
  In particular,
    \begin{align*}
      \|(v,\ssigma)\|_{L^2(Q)\times \bL^2(Q)}^2  \lesssim \|\cA(v,\ssigma)\|_{L^2(Q)^{1+d}\times L^2(\Omega)^{1+d}}
      \quad\forall (v,\ssigma)\in V_0.
    \end{align*}
\end{theorem}
\begin{proof}
  The operator $\cA$ is linear and continuous.
  Combining Lemmas~\ref{lem:FOS:surj:1},~\ref{lem:FOS:surj:2},~\ref{lem:FOS:surj:3} and~\ref{lem:FOS:inj}
  we conclude that $\cA$ is bijective. A direct consequence of Banach's \emph{Open Mapping Theorem}
  is that the inverse of $\cA$ is continuous.
\end{proof}
Our second main results establishes that the space of smooth functions with homogeneous Dirichlet boundary
conditions is dense in $V_0$. The proof will be given in section~\ref{sec:density} below.
\begin{theorem}\label{thm:density}
  Space $\boldsymbol{C}^\infty := \{(v,\ssigma)\in C^\infty(\overline Q)^{1+d}\mid v|_{J\times \partial\Omega} = 0\}$ is dense in $V_0$.
\end{theorem}
\subsection{Analysis of the first-order wave operator}
In the following three lemmas we show that the operator $\cA$ is surjective. 
\begin{lemma}\label{lem:FOS:surj:1}
  Let $f\in L^2(Q)$, $u_0\in H^1_0(\Omega)$, $u_1\in L^2(\Omega)$ be given.
  There exists $(v,\ssigma)\in V_0$ such that
  \begin{align*}
    \cA(v,\ssigma) = ( f,0,u_1,\nablasp u_0)
  \end{align*}
  and 
  \begin{align*}
    \|(v,\ssigma)\|_{V} \lesssim \|f\|_{L^2(Q)} + \|\nablasp u_0\|_{L^2(\Omega)} + \|u_1\|_{L^2(\Omega)}.
  \end{align*}
\end{lemma}
\begin{proof}
  First, let $u\in L^2(H_0^1(\Omega))\cap H^1(L^2(\Omega))\cap H^2(H^{-1}(\Omega))$ denote the unique solution
  of the wave equation~\eqref{eq:wave} and define
  $(v,\ssigma) = (\partial_t u,\nablasp u)\in L^2(L^2(\Omega))\times L^2(\bL^2(\Omega))$.
  Furthermore, $\stdiv(v,-\ssigma)\in L^2(Q)$
  and $\partial_t\ssigma = \nablasp v\in L^2(\bH^{-1}(\Omega))$.
  Next, note that
  $v(0)=\partial_t u(0) = u_1\in L^2(\Omega)$ and, due to Lemma~\ref{lem:schwarz} (i),
  $\ssigma(0) = (\nablasp u)(0) = \nablasp (u(0)) = \nablasp u_0\in L^2(\Omega)$.
  We conclude that $(v,\ssigma)\in V$.
  It remains to show that $(v,\ssigma)\in \ker D$.
  Let $\varphi \in Y$ be given. 
  Using $\ssigma = \nablasp u$, $v = \partial_t u$, $\partial_t\ssigma-\nablasp v = 0$ we obtain that
  \begin{align}\label{proof:surj1:a}
  \begin{split}
    D(v,\ssigma)(\varphi) 
     &= \int_0^T \vdual{\nablasp u}{\partial_t\nablasp\varphi}_\Omega
     - \vdual{\partial_t u}{\Deltasp\varphi}_\Omega\,ds + \vdual{\ssigma(0)}{\nablasp\varphi(0)}_\Omega.
   \end{split}
  \end{align}
  Note that $\partial_t\varphi \in L^2(H_0^1(\Omega))$ and $\Deltasp\varphi\in L^2(Q)$. Therefore, $\Deltasp\partial_t\varphi = \partial_t\Deltasp\varphi \in L^2(H^{-1}(\Omega))$ and $\Deltasp\varphi(t)\in H^{-1}(\Omega)$ for all $t\in \overline J$.
  Furthermore, $u(0) \in H_0^1(\Omega)$ and $u(T)\in H_0^1(\Omega)$. One verifies the integration by parts formula
  \begin{align}\label{proof:surj1:b}
  \begin{split}
    \int_0^T \vdual{\partial_t u}{\Deltasp\varphi}_\Omega\,ds 
     &= \int_0^T \vdual{\nablasp u}{\partial_t \nablasp\varphi}_\Omega \,ds + \vdual{\nablasp u(0)}{\nablasp\varphi(0)}_\Omega.
   \end{split}
  \end{align}
  Here, we have used that $\varphi(T) = 0$, thus, $\Deltasp\varphi(T) = 0$. Combining~\eqref{proof:surj1:a} and~\eqref{proof:surj1:b} and using that $\nablasp u(0) = \ssigma(0)$ we conclude that $D(v,\ssigma)(\varphi) = 0$ for all $\varphi\in Y$ which proves that $(v,\ssigma)\in V_0$.

  The stability estimate is a direct consequence of the stability of the wave problem~\eqref{eq:wave}
  with $v=\partial_t u$, $\ssigma = \nablasp u$.
\end{proof}

\begin{lemma}\label{lem:FOS:surj:2}
  Let $\g \in L^2(\nablasp(H_0^1(\Omega)))$ be given. There exists $(v,\ssigma)\in V_0$ such that
  \begin{align*}
    \cA(v,\ssigma) = (0,\g,0,0) \quad\text{and}\quad \|(v,\ssigma)\|_{V} \lesssim \|\g\|_{\bL^2(Q)}.
  \end{align*}
\end{lemma}
\begin{proof}
  Set $f=\divsp\g\in L^2(H^{-1}(\Omega))$, $g_0=0$, $g_1 = 0$ and let $u\in L^2(Q)\cap H^1(H^{-1}(\Omega))$ be the solution of the ultraweak wave equation from Theorem~\ref{lm:wave:uw}. Let $\ssigma\in L^2(\nablasp H_0^1(\Omega))$ denote the unique element with $\divsp\ssigma = \partial_t u$.
  We have $(v,\ssigma):=(u,\ssigma)\in L^2(Q)^{1+d}$,
  and also $\partial_t v-\divsp\ssigma=0\in L^2(L^2(\Omega))$. Next, we show that
  $\partial_t\ssigma-\nablasp v=\g \in L^2(\bL^2(\Omega))$, which by density
  of $\DD(J)\otimes \DD(\Omega)^d\subset L^2(\bL^2(\Omega))$ and linearity, boils down to showing that for all
  $\varphi\in \DD(J)$ and $\xxi\in\DD(\Omega)^d$ it holds
  \begin{align*}
    \int_J \vdual{\g}{\xxi}_\Omega \varphi\,ds &=
    \int_J -\vdual{\ssigma}{\xxi}_\Omega\partial_t\varphi + \vdual{v}{\divsp\xxi}_\Omega \varphi\,ds.
  \end{align*}
  Writing $\xxi= \Pi_{\nablasp H^1_0(\Omega)}\xxi + \Pi_{\nablasp H^1_0(\Omega)^\perp} \xxi$, taking into account
  that $\g,\ssigma \in L^2(\nablasp H^1_0(\Omega))$ and $\divsp\Pi_{\nablasp H^1_0(\Omega)^\perp} \xxi=0$,
  it remains to show that
  \begin{align*}
    \int_J \vdual{\g}{\Pi_{\nablasp H^1_0(\Omega)}\xxi}_\Omega \varphi\,ds &=
    \int_J -\vdual{\ssigma}{\Pi_{\nablasp H^1_0(\Omega)}\xxi}_\Omega\partial_t\varphi
    + \vdual{v}{\divsp\Pi_{\nablasp H^1_0(\Omega)}\xxi}_\Omega \varphi\,ds.
  \end{align*}
  Take $\chi\in H^1_0(\Omega)$ such that $\nablasp \chi = \Pi_{\nablasp H^1_0(\Omega)}\xxi$
  and note that $\phi:= \varphi \chi$ fulfills
    \begin{align*}
      \partial_{tt}\phi = \partial_{tt}\varphi\chi \in L^2(L^2(\Omega))
      \text{ and } \Deltasp\phi = \varphi \divsp\Pi_{\nablasp H^1_0(\Omega)}\xxi
      = \varphi \divsp\xxi \in L^2(L^2(\Omega)).
    \end{align*}
    As $v$ solves the ultraweak wave equation, $\phi\in\xX$ and $\phi(0)=0$,
  \begin{align*}
    &\int_J -\vdual{\ssigma}{\Pi_{\nablasp H^1_0(\Omega)}\xxi}_\Omega\partial_t\varphi
    + \vdual{v}{\divsp\Pi_{\nablasp H^1_0(\Omega)}\xxi}_\Omega \varphi\,ds
    \\ &\qquad= 
    \int_J -\vdual{\ssigma}{\nablasp\chi}_\Omega\partial_t\varphi +
    \vdual{v}{\divsp\xxi}_\Omega \varphi\,ds\\
    &\qquad=
    -\int_0^T \vdual{v}{\partial_{tt}\phi-\Deltasp \phi}_\Omega\,ds =
    \int_0^T \vdual{\g}{\nablasp\phi}_\Omega\,ds
    = \int_J \vdual{\g}{\Pi_{\nablasp H^1_0(\Omega)}\xxi}_\Omega \varphi\,ds.
  \end{align*}
  Theorem~\ref{lm:wave:uw} gives $v(0)=0\in L^2(\Omega)$, and
  Lemma~\ref{lem:trace} additionally shows that $\divsp \ssigma(0) = \partial_t u(0)=0$. Furthermore, as
  $\ssigma(0)$ is the gradient of an $H^1_0(\Omega)$-function, we conclude $\ssigma(0)=0$.
  It remains to prove that $(v,\ssigma)\in V_0$.
  Let $\varphi\in Y$ be given. Using that $\partial_t \varphi \in L^2(H_0^1(\Omega))$,
  $\divsp\ssigma(0)=0$, $\divsp\ssigma = \partial_t v$, $\partial_{tt}\varphi \in L^2(Q)$, $v(0) = 0$, and $\partial_t\varphi(T) = 0$ we obtain that
  \begin{align}\label{proof:surj2:a}
  \begin{split}
    \int_0^T \vdual{\ssigma}{\partial_t\nablasp\varphi}_\Omega\,ds &= -\int_0^T \dual{\divsp\ssigma}{\partial_t \varphi}_{H^{-1}(\Omega)\times H_0^1(\Omega)} 
    \\ &= -\int_0^T \dual{\partial_t v}{\partial_t \varphi}_{H^{-1}(\Omega)\times H_0^1(\Omega)} \\
    &= \int_0^T \vdual{v}{\partial_{tt}\varphi}_\Omega \,ds + \vdual{\partial_t\varphi(0)}{v(0)}_\Omega-\vdual{\partial_t\varphi(T)}{v(T)}_\Omega 
    \\
    &= \int_0^T \vdual{v}{\partial_{tt}\varphi}_\Omega \,ds.
  \end{split}
  \end{align}
  Recall that $v$ solves the ultraweak wave equation so that in particular
  \begin{align*}
    \int_0^T \vdual{v}{\partial_{tt}\varphi-\Deltasp\varphi}_\Omega \,ds = \int_0^T-\vdual{\g}{\nablasp\varphi}_\Omega\,ds
  \end{align*}
  Combining this with~\eqref{proof:surj2:a}, $\partial_t\ssigma-\nablasp v=\g$ and $\ssigma(0)=0$ we further conclude that
  \begin{align*}
      D(v,\ssigma)(\varphi) 
    &= \int_0^T \vdual{\g}{\nablasp\varphi}_\Omega + \vdual{\ssigma}{\nablasp\partial_t\varphi}_\Omega
    - \vdual{v}{\Deltasp \varphi}_\Omega\,ds
    \\
    &= \int_0^T \vdual{\g}{\nablasp\varphi}_\Omega + \vdual{v}{\partial_{tt}\varphi-\Deltasp \varphi}_\Omega\,ds = 0.
  \end{align*}
  This proves that $(v,\ssigma)\in V_0$.
  By the stability properties of the ultra-weak wave equation and the fact that $\|\ssigma\|_{\bL^2(Q)} \eqsim \|\divsp\ssigma\|_{L^2(H^{-1}(\Omega))}$ we get that
  \begin{align*}
    \|v\|_{L^2(Q)} + \|\ssigma\|_{\bL^2(Q)} &\eqsim \|u\|_{L^2(Q)}+\|\divsp\ssigma\|_{L^2(H^{-1}(\Omega))} \!=\! \|u\|_{L^2(Q)}\!+\!\|\partial_t u\|_{L^2(H^{-1}(\Omega))} 
    \\
    &\lesssim \|\divsp\g\|_{L^2(H^{-1}(\Omega))} \eqsim \|\g\|_{\bL^2(Q)}.
  \end{align*}
  From there we get that $\|(v,\ssigma)\|_{V} \lesssim \|\g\|_{\bL^2(Q)}$ which finishes the proof.
\end{proof}
So far we have only shown that if the vector-valued datum on the right-hand side can be represented
as gradients of $H^1_0(\Omega)$-functions, then they are in the range of operator $\cA$. In the next result we complement on that.
\begin{lemma}\label{lem:FOS:surj:3}
  Let $\g^\perp \in L^2(\nablasp H^1_0(\Omega)^\perp)$ and
  $\ssigma_0^\perp\in \nablasp H^1_0(\Omega)^\perp$ be given.
  There exists $(v,\ssigma)\in V_0$ such that
  \begin{align*}
    \cA(v,\ssigma) = (0,\g^\perp,0,\ssigma_0^\perp) \quad\text{and}\quad
    \|(v,\ssigma)\|_{V} \lesssim \|\g^\perp\|_{L^2(Q)} + \|\ssigma_0^\perp\|_{L^2(\Omega)}.
  \end{align*}
\end{lemma}
\begin{proof}
  Set $v:=0$ and $\ssigma(t) := \ssigma_0^\perp + \int_0^t \g^\perp(s)\,ds$.
  Then, as $\divsp\ssigma_0^\perp=0$ and $\divsp\g^\perp=0$, we conclude $\partial_t v-\divsp\ssigma=\divsp\ssigma=0$.
  Furthermore, $\partial_t\ssigma-\nablasp v = \partial_t\ssigma= \g^\perp$. Also, $v(0)=0\in L^2(\Omega)$ and
  $\ssigma(0)=\ssigma_0^\perp\in \bL^2(\Omega)$. 
  It remains to show that $(v,\ssigma)\in V_0$. Let $\varphi\in Y$ be given.
  Using $\nablasp\varphi(T) = 0$ we obtain with integration by parts that
  \begin{align*}
    D(u,\ssigma)(\varphi)= \int_0^T \vdual{\partial_t\ssigma}{\nablasp\varphi}_\Omega + \vdual{\ssigma}{\partial_t\nablasp\varphi}_\Omega\,ds
    + \vdual{\ssigma(0)}{\nablasp\varphi(0)}_\Omega=0.
  \end{align*}
  Therefore, $(v,\ssigma)\in V_0$.
  Finally, the stability estimate
  follows from the given representation of $\ssigma$ and $v=0$.
\end{proof}
In the next result we prove that operator $\cA$ is also injective.
\begin{lemma}\label{lem:FOS:inj}
  If $\cA(v,\ssigma) = 0$ for some $(v,\ssigma)\in V_0$, then $(v,\ssigma) = 0$.
\end{lemma}
\begin{proof}
  Let $(v,\ssigma)\in V_0$ be given with $\cA(v,\ssigma)=0$, i.e., 
  \begin{align*}
    \partial_t v-\divsp\ssigma &= 0 = \partial_t \ssigma -\nablasp v, \\
    v(0) &= 0 = \ssigma(0).
  \end{align*}
  In what follows we show that $(v,\ssigma)$ satisfies the ultra-weak wave equation with vanishing data and, thus, $(v,\ssigma)=0$ by Theorem~\ref{lm:wave:uw}.
  Let $\varphi\in \xX$ with $\Deltasp \varphi\in L^2(Q)$. 
  Set $g:=\partial_{tt}\varphi-\Delta_x \varphi$ and let $g_n \in C_0^\infty(Q)$ be a sequence with $g_n\to g$ in $L^2(Q)$. Define $\varphi_n\in \xX$ to be the unique solution of the wave equation
  \begin{align*}
    \partial_{tt} \varphi_n - \Deltasp \varphi_n &= g_n \quad\text{in } \Omega, \\
    \varphi_n|_{(0,T)\times\partial\Omega} &= 0, \\
    \varphi_n(T) &= 0 = \partial_t \varphi_n(T).
  \end{align*}
  By regularity estimates, see Lemma~\ref{lem:regularity}, we obtain that
  \begin{align*}
    \partial_t \varphi_n \in L^\infty(H_0^1(\Omega)), \quad 
    \partial_{tt} \varphi_n \in L^\infty(L^2(\Omega)).
  \end{align*}
  In particular, $\partial_{tt}\varphi_n \in L^2(Q)$, thus, $\Deltasp\varphi_n = \partial_{tt}\varphi_n-g_n\in L^2(Q)$.
  Integration by parts with $\partial_t\varphi_n|_{\partial Q\setminus \{0\}\times\Omega} = 0$ and $v(0)=0$ yields
  \begin{align}\label{eq:proof:inj1}
  \begin{split}
    \int_0^T \vdual{v}{\partial_{tt}\varphi_n}_\Omega -\vdual{\ssigma}{\nablasp \partial_t \varphi_n}_\Omega \,ds 
    &= \int_0^T \vdual{(v,-\ssigma)}{\stnabla \partial_t\varphi_n}_\Omega \,ds 
    \\
    &= -\int_0^T \vdual{\stdiv(v,-\ssigma)}{\varphi_n}_\Omega\, ds =0.
  \end{split}
  \end{align}
  Note that by the aforegoing observations we have that $\varphi_n \in Y$.
  Since $(v,\ssigma)\in \ker D$, $\partial_t\ssigma-\nablasp v=0$, $\ssigma(0)=0$ we get that
  \begin{align}\label{eq:proof:inj2}
  \begin{split}
    0&=D(v,\ssigma)(\varphi_n) 
    = \int_0^T \vdual{\nablasp\partial_t\varphi_n}{\ssigma}_\Omega-\vdual{\Deltasp\varphi_n}v_\Omega \, ds.
  \end{split}
  \end{align}
  Using~\eqref{eq:proof:inj1}-\eqref{eq:proof:inj2} we see that
  \begin{align*}
    \int_0^T \vdual{v}{\partial_{tt}\varphi_n-\Deltasp\varphi_n}_\Omega \, ds 
    &= \int_0^T \vdual{\ssigma}{\nablasp\partial_t\varphi_n}_\Omega-\vdual{v}{\Deltasp\varphi_n}_\Omega\, ds 
    \\
    &= \int_0^T \vdual{\Deltasp\varphi_n}{v}_\Omega - \vdual{v}{\Deltasp\varphi_n}_\Omega\, ds = 0.
  \end{align*}
  Taking the limit $n\to\infty$ we infer that
  \begin{align*}
    \int_0^T \vdual{v}{\partial_{tt}\varphi-\Deltasp\varphi}_\Omega \,ds = 0 \quad\text{for all } \varphi\in \xX \text{ with } \Deltasp\varphi \in L^2(Q).
  \end{align*}
  By virtue of Theorem~\ref{lm:wave:uw} we conclude that $v=0$. Finally, $\partial_t \ssigma = \nablasp v = 0$ together with $\ssigma(0)=0$ gives $\ssigma = 0$, thus, $(v,\ssigma)=0$ which finishes the proof of injectivity.
\end{proof}
\subsection{Proof of Theorem~\ref{thm:density}}\label{sec:density}
Density is based on the following definitions and observations which are related to
the wave equation with terminal conditions instead of initial conditions.
Analogously as for the definition of the space $V$, we may invoke Lemma~\ref{lem:trace} to define
the trace operator $\gamma_T$ at final time $T$, and the spaces
\begin{align*}
  V_T &:= \left\{ (v,\ssigma)\in \wilde V \mid \gamma_T(v,\ssigma)\in L^2(\Omega)\times \bL^2(\Omega) \right\}\\
  Y_T &= \left\{\varphi \in H^1(H_0^1(\Omega))\cap H^2(L^2(\Omega)) \cap L^2(H(\Deltasp;\Omega))\mid \varphi(0) = 0,\, \partial_t\varphi(0) = 0\right\},
\end{align*}
where $V_T$ and $Y_T$ are equipped with their respective graph norms.
Then, we define the linear and continuous operator $D_T\colon V_T\to Y_T'$ by
\begin{align*}
  D_T(v,\ssigma)(\varphi) = \vdual{\partial_t\ssigma-\nablasp v}{\nablasp\varphi}_Q + \vdual{\ssigma}{\nablasp\partial_t\varphi}_Q
  - \vdual{v}{\Deltasp \varphi}_Q - \vdual{\ssigma(T)}{\nablasp\varphi(T)}_{\Omega}
\end{align*}
for $(v,\ssigma)\in V_T$, $\varphi\in V_T$.

We start with a different presentation of graph space $\wilde V$. Introducing the operator $B\colon L^2(Q)^{1+d} \to L^2(Q)^{(1+d)\times (1+d)}$ we write 
\begin{align*}
  B(v,\ssigma) = \begin{pmatrix}
    v & -\ssigma_1 & \cdots & \cdots & -\ssigma_d \\
    \ssigma_1 & -v & 0 & \cdots & 0 \\
    \vdots & 0 & -v & 0 & \vdots \\
    \vdots & \ddots & \ddots & \ddots & \vdots \\
    \ssigma_d & 0 & \cdots & & -v
  \end{pmatrix}
\end{align*}
Then, operator $\cA_0$ formally reads $\cA_0(v,\ssigma) = \stDiv B(v,\ssigma)$. Clearly, 
\begin{align*}
  \wilde V = \{ (v,\ssigma)\in L^2(Q)^{1+d}\mid \stDiv B(v,\ssigma) \in L^2(Q)^{1+d}\}.
\end{align*}
Here $\stDiv$ denotes the $\stdiv$ operator applied row-wise. Further, for elements in $\wilde V$ we can define a trace (in the sense of normal traces), i.e., 
\begin{align*}
  \dual{B(v,\ssigma)\nn}{\pphi}_{\partial Q} = \vdual{\stDiv B(v,\ssigma)}{\pphi}_Q + \vdual{B(v,\ssigma)}{\stNabla\pphi}_Q \quad\forall \pphi\in H^1(Q)^{1+d}.
\end{align*}
Here, the $j$th row of $\stNabla \pphi$ is $\stnabla(\pphi_j)$. A simple computation yields that 
\begin{align*}
  \vdual{B(v,\ssigma)}{\stNabla\pphi}_Q = \vdual{(v,\ssigma)}{\stDiv B\pphi}_Q.
\end{align*}

\begin{proof}[Proof of Theorem~\ref{thm:density}]
  We will prove a slightly different result: Let $H^1_{J\times \partial\Omega}(Q)$ be the space of
functions in $H^1(Q)$ vanishing on $J\times\partial\Omega\subset\partial Q$. We will show that
    $W := H_{J\times \partial\Omega}^1(Q) \times (L^2(\bH(\divsp;\Omega))\cap H^1(\bL^2(\Omega)))$
    is dense in $V_0$. Note that we have the continuous embedding $W\hookrightarrow V_0$, and
    the densities of 
    \begin{align*}
      \{v\in C^\infty(\overline Q)\mid v|_{J\times\partial\Omega}=0\} &\quad\text{ in }\quad
      H_{J\times \partial\Omega}^1(Q),\\
      C^\infty(\overline J)\otimes C^\infty(\overline \Omega)^d &\quad\text{ in }\quad
      L^2(\bH(\divsp;\Omega))\cap H^1(\bL^2(\Omega)),
    \end{align*}
    where the second density result follows from Lemma~\ref{lem:density} (ii)
  and the well-known density of $C^\infty(\overline\Omega)^d$ in $\bH(\divsp;\Omega)$.
It is therefore enough to prove that $W$ is dense in $V_0$.

  Density of $W$ is equivalent to showing that for each $\ell\in (V_0)'$ with $\ell( (v,\ssigma)) = 0$ for all $(v,\ssigma)\in W$ implies that $\ell = 0$.
  Let $\ell\in (V_0)'$ be given. The Riesz representation theorem and Theorem~\ref{thm:main} imply
  that there exists a unique $(w,\cchi)\in V_0$ such that 
  \begin{align}\label{eq:proof:dens1}
    \ell((v,\ssigma)) = \vdual{\cA_0(w,\cchi)}{\cA_0(v,\ssigma)}_Q + \vdual{(w(0),\cchi(0))}{(v(0),\ssigma(0))}_\Omega 
  \end{align}
  for all $(v,\ssigma) \in V_0$.
  
  Suppose that $\ell((v,\ssigma)) = 0$ for all $(v,\ssigma) \in W$. 
  Take $(v,\ssigma) \in \mathcal{D}(Q)^{1+d}$ and set $(u,\ttau) = \cA_0(w,\cchi)\in L^2(\Omega)^{1+d}$. The definition of weak derivatives implies that 
  \begin{align*}
    0 = \ell((v,\ssigma)) = \vdual{(u,\ttau)}{\cA_0(v,\ssigma)}_Q = -\dual{\cA_0(u,\ttau)}{(v,\ssigma)}_{\cD'(Q)\times \mathcal{D}(Q)}.
  \end{align*}
  Therefore, $L^2(Q)^{1+d}\ni\cA_0(u,\ttau) = 0$. This means that $(u,\ttau)\in \wilde V$.

  Next, we prove that $(u(T),\ttau(T))=0$.
  To that end take $(v,\ssigma)\in [C^\infty(\overline J)\otimes \mathcal{D}(\Omega)]^{1+d} \subset \boldsymbol{C}^\infty \subset W$ with $(v(0),\ssigma(0)) = 0$. Observe that
  \begin{align*}
    0 = \ell( (v,\ssigma) ) &= \vdual{(u,\ttau)}{\stDiv B(v,\ssigma)}_Q \\
    &= -\vdual{\stDiv B(u,\ttau)}{(v,\ssigma)}_Q + \dual{B(u,\ttau)\nn}{(v,\ssigma)}_{\partial Q} \\
    &= \dual{B(u,\ttau)\nn}{(v,\ssigma)}_{\{T\}\times \Omega}
  \end{align*}
  where we have used the already established fact that $\cA_0(u,\ttau)=0$,
  and $v,\ssigma$ vanish on $\partial Q \setminus \{T\}\times\Omega$.
  We conclude that $u(T) = 0$ as well as $\ttau(T) = 0$. 

  In what follows we show that $(u,\ttau)\in \ker D_T$.
  Take $\varphi\in Y_T$ and note that $(0,\nablasp\varphi)\in W$.
  Then, using that $\vdual{(u,\ttau)}{\stDiv B(0,\nablasp\varphi)}_Q = 0$ (by~\eqref{eq:proof:dens1}) and $\ttau(T)=0$, $\stDiv B(u,\ttau) = 0$ we infer that
  \begin{align*}
    D_T(u,\ttau)(\varphi) &= \vdual{\partial_t\ttau-\nablasp u}{\nablasp \varphi}_Q + \vdual{\ttau}{\partial_t\nablasp\varphi}_Q -\vdual{u}{\Deltasp\varphi}_Q -\vdual{\ttau(T)}{\nablasp\varphi(T)}_\Omega \\
    &= \vdual{\stDiv B(u,\ttau)}{(0,\nablasp\varphi)}_Q + \vdual{(u,\ttau)}{\stDiv B(0,\nablasp\varphi)}_Q = 0.
  \end{align*}
  This proves that $D_T(u,\ttau)(\nablasp\varphi) = 0$ for all $\varphi \in Y_T$. 

  To conclude the proof we note that $(u,\ttau)\in \wilde V$ with $u(T) = 0$, $\ttau(T) = 0$ and $(u,\ttau)\in \ker D_T$. Arguing as in Lemma~\ref{lem:FOS:inj} (interchanging initial and terminal conditions) yields $(u,\ttau) = 0$, and, therefore, $\cA_0(w,\cchi) = (u,\ttau) = 0$. Then, 
  \begin{align*}
    0 = \ell((v,\ssigma)) &= \vdual{\cA_0(w,\cchi)}{\cA_0(v,\ssigma)}_Q + \vdual{(w(0),\cchi(0))}{(v(0),\ssigma(0))}_\Omega
    \\
    &= \vdual{(w(0),\cchi(0))}{(v(0),\ssigma(0))}_\Omega \quad\forall (v,\ssigma)\in [C^\infty(\overline J)\otimes \mathcal{D}(\Omega)]^{1+d} \subset \boldsymbol{C}^\infty.
  \end{align*}
  It follows that $(w(0),\cchi(0)) = 0$. This together with $\cA_0(w,\cchi) = 0$ and~\eqref{eq:proof:dens1} proves $\ell = 0$ which finishes the proof.   
\end{proof}

\section{A least-squares finite element method}\label{sec:LS}
In this section we define and analyze a least-squares finite element method based on Theorem~\ref{thm:main}.
Then, we discuss an adaptive algorithm which we employ for our numerical experiments. 

\subsection{Least-squares formulation and discretization}
Given $f\in L^2(Q)$, $\g \in\bL^2(Q)$, $v_0\in L^2(\Omega)$, $\ssigma_0\in \bL^2(\Omega)$, define the (quadratic) functional $G\colon U\to\R$ by
\begin{align*}
  G(v,\ssigma;f,\g,v_0,\ssigma_0) &= \|\cA(v,\ssigma)-(f,\g,v_0,\ssigma_0)\|_{L^2(Q)^{1+d}\times L^2(\Omega)^{1+d}}^2\\
  &=
  \|\partial_t v-\divsp\ssigma-f\|_Q^2 + \|\partial_t\ssigma-\nablasp v-\g\|_Q^2
  \\
  &\qquad \qquad +\|v(0)-v_0\|_\Omega^2 + \|\ssigma(0)-\ssigma_0\|_\Omega^2.
\end{align*}
It is immediate that a solution to the first-order system wave equation, i.e, $\cA(v,\ssigma)=(f,\g,v_0,\ssigma_0)$ is a minimizer of $G$. 
Note that Theorem~\ref{thm:main} implies the norm equivalence
\begin{align}\label{eq:LS:normeq}
  G(v,\ssigma;0,0,0,0) \eqsim \|(v,\ssigma)\|_V^2 \quad\forall (v,\ssigma)\in V_0.
\end{align}
Let $V_h\subseteq V_0$ denote a closed subspace and consider the minimization problem
\begin{align}\label{eq:LS:discrete}
  (v_h,\ssigma_h) = \operatorname*{arg\,min}_{(w,\cchi)\in V_h} G(v_h,\ssigma_h;f,\g,v_0,\ssigma_0)
\end{align}
Introducing the bilinear form $a\colon V_0\times V_0\to \R$, and the linear form $F\colon V_0\to \R$, 
\begin{align*}
  a(v,\ssigma;w,\cchi) &= \vdual{\partial_t v-\divsp\ssigma}{\partial_t w-\divsp\cchi}_Q 
  + \vdual{\partial_t\ssigma-\nablasp v}{\partial_t\cchi-\nablasp w}_Q \\
  &\qquad + \vdual{v(0)}{w(0)}_\Omega + \vdual{\ssigma(0)}{\cchi(0)}_\Omega, \\
  F(v,\ssigma) &= \vdual{f}{\partial_t w-\divsp\cchi}_Q + \vdual{\g}{\partial_t\cchi-\nablasp w}_Q
  \\
  &\qquad + \vdual{v_0}{w(0)}_\Omega + \vdual{\ssigma_0}{\cchi(0)}_\Omega
\end{align*}
for all $(v,\ssigma), (w,\cchi)\in V_0$, the Euler--Lagrange equations for problem~\eqref{eq:LS:discrete} read: Find $(v_h,\ssigma_h)\in V_h$ such that
\begin{align}\label{eq:LS:EL}
  a(v_h,\ssigma_h;w_{h},\cchi_{h}) = F(w_{h},\cchi_{h}) \quad\forall
  (w_{h},\cchi_{})\in V_h.
\end{align}
Based on the norm equivalence~\eqref{eq:LS:normeq}, the theory on least-squares finite element methods, see, in particular,~\cite[Chapter~3, Theorem~2.5]{BochevG_09} implies the following result. 
\begin{theorem}\label{thm:LS}
  Let $f\in L^2(Q)$, $\g\in \bL^2(Q)$, $v_0\in L^2(\Omega)$, $\ssigma_0\in \bL^2(\Omega)$ be given. Problems~\eqref{eq:LS:discrete} and~\eqref{eq:LS:EL} are equivalent and admit a unique solution $(v_h,\ssigma_h)\in V_h$. 

  Let $(v,\ssigma)\in V_0$ denote the unique solution of
  \begin{align*}
    \cA(v,\ssigma) = (f,\g,v_0,\ssigma_0).
  \end{align*}
  Then, there exists a constant $C>0$ only depending on the norm equivalence constants from~\eqref{eq:LS:normeq} such that
  \begin{align*}
    \|(v,\ssigma)-(v_h,\ssigma_h)\|_V \leq C \min_{(w_{h},\cchi_{h})\in V_h}
    \|(v,\ssigma)-(w_{h},\cchi_{h})\|_V.
  \end{align*}
\end{theorem}

In this work we study a simple space-time discretization based on piecewise polynomials that are globally continuous. 
The motivation is simply a practical one since such finite element spaces are implemented in almost all finite element software packages.
Let $\cT$ denote a decomposition of $Q = J\times \Omega$ into open simplices in $\R^{1+d}$. Let $h_\cT\in L^\infty(Q)$ denote the local mesh-width function given by $h_\cT|_T = \diam(T)$. 
Let $P^p(T)$ denote the space of polynomials of degree $\leq p\in \N_0$ on $T\in\cT$ and set
\begin{align*}
  P^p(\cT) &= \{v\in L^2(Q)\mid v|_T \in P^p(T) \, \forall T\in\cT\}, \qquad S^p(\cT) = P^p(\cT) \cap H^1(Q), \\
S_\Gamma^p(\cT) &= S^p(\cT)\cap H_{J\times \partial\Omega}^1(Q) \quad\text{and}\quad V_{\cT,p} = S_\Gamma^p(\cT)\times [S^p(\cT)]^d.
\end{align*}
By Lemma~\ref{lem:bcdirichlet} we have that $V_{\cT,p}\subset V_0$.
Our LSFEM for the wave equation is then given by~\eqref{eq:LS:EL} with $V_h = V_{\cT,p}$.

A direct consequence of Theorem~\ref{thm:LS} and standard approximation results, cf.~\cite{BrennerScott}, is the following result.
\begin{corollary}\label{cor:rates}
  Let $p\in \N$ and set $V_h = V_{\cT,p}$. Using the notation of Theorem~\ref{thm:LS},
  suppose that the exact solution satisfies $(v,\ssigma)\in H^{p+1}(Q)^{1+d}$. Then, 
  \begin{align*}
    \|(v,\ssigma)-(v_h,\ssigma_h)\|_V \lesssim \|h_\cT\|_{L^\infty(Q)}^{p} \|(v,\ssigma)\|_{H^{p+1}(Q)^{1+d}}.
  \end{align*}
\end{corollary}

\subsection{A posteriori error estimator and adaptive algorithm}
Let $(v_\cT,\ssigma_\cT)$ denote the solution of~\eqref{eq:LS:EL} with $V_h= V_{\cT,p}$.
Define for each $T\in\cT$ the error indicator
\begin{align*}
  \eta_T^2 &= \|\partial_t v-\divsp\ssigma-f\|_T^2 + \|\partial_t\ssigma-\nablasp v-\g\|_T^2
  \\
  &\qquad \qquad +\|v(0)-v_0\|_{\partial T\cap \Omega}^2 + \|\ssigma(0)-\ssigma_0\|_{\partial T\cap\Omega}^2
\end{align*}
and the overall estimator $\eta_\cT = \sqrt{\sum_{T\in\cT}\eta_T^2}$.

Using that $\cA(v,\ssigma) = (f,\g,v_0,\ssigma_0)$ and Theorem~\ref{thm:main} we obtain
\begin{align*}
  \eta_\cT^2 = \|\cA(v_\cT-v,\ssigma_\cT-\ssigma)\|_{L^2(Q)^{1+d}\times L^2(\Omega)^{1+d}}^2 \eqsim \|(v,\ssigma)-(v_\cT,\ssigma_\cT)\|_V^2.
\end{align*}
This means that $\eta_\cT$ is a reliable and efficient error estimator. 

We consider the standard adaptive loop
\begin{itemize}
  \item \textbf{Solve} problem~\eqref{eq:LS:EL},
  \item \textbf{Estimate} (compute error indicators $\eta_T$), 
  \item \textbf{Mark} elements of $\cT$ for refinement, 
  \item \textbf{Refine} at least all marked elements to obtain a new mesh $\cT$.
\end{itemize}
The marking step is realized via the bulk criterion: Given a parameter $\theta \in (0,1)$, find a (minimal) set of marked elements $\mathcal{M}\subset \cT$ with 
\begin{align*}
  \theta \eta_\cT^2 \leq \sum_{T\in\mathcal{M}} \eta_T^2.
\end{align*}

\begin{remark}\label{rem:convergence}
  Under some assumptions on mesh-refinement, discrete spaces and marking, one can show that the solutions generated by the adaptive algorithm converge to the exact solution, see~\cite{MR4138307} or~\cite{GantnerS_21}.
\end{remark}
\section{Numerical examples}\label{sec:num}
In this section we present some examples for $d=1$ and $d=2$.
We implemented the proposed LSFEM in \textsc{NETGEN/NGSOLVE} \cite{ngsolve}.
Our implementation can be found following the link \verb+github.com/tofuuhh/LSQwave+.

For all experiments we use the unit $(1+d)$-cube $Q = J\times \Omega = (0,1)\times (0,1)^d$ and choose $\theta = \tfrac14$ if the adaptive algorithm is employed.

\subsection{Smooth solution for $d=1$}
We consider $u(t,x) = \tfrac12 t^2 \sin(\pi x)$. Then, $v = \partial_t u$, $\ssigma = \partial_x u$ solve the first-order system wave equation with data $f= \partial_{tt}u - \partial_{xx} u$, $\g = v_0=\ssigma_0 = 0$.
On a sequence of uniform refinements of an initial triangulation of $Q$, Corollary~\ref{cor:rates} predicts rates which are perfectly aligned with the obtained results, see Figure~\ref{fig:1dsmooth}.

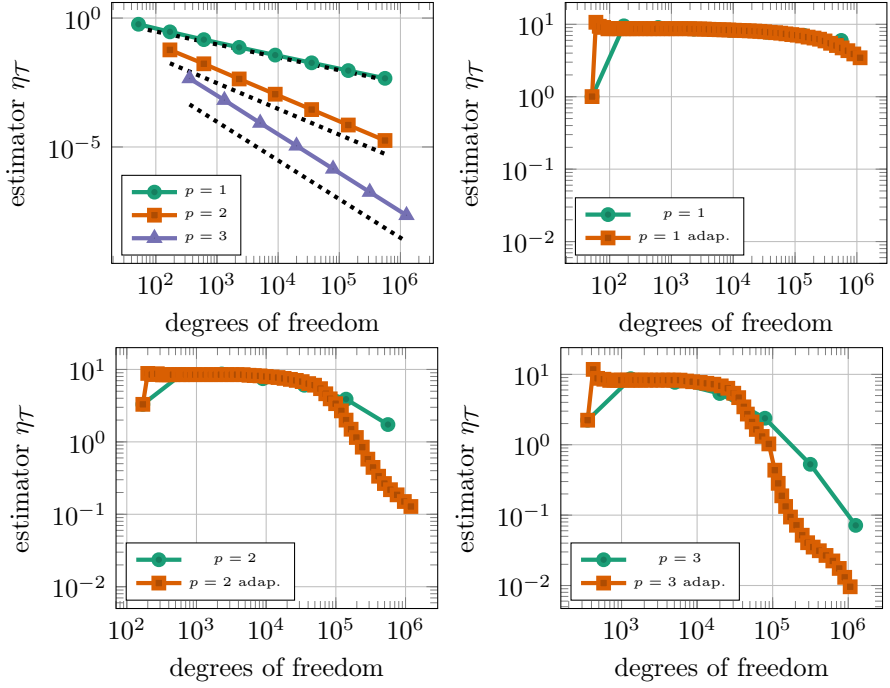
\begin{figure}
  \begin{center}
    \begin{tikzpicture}
\begin{loglogaxis}[
    width=0.45\textwidth,
cycle list/Dark2-6,
cycle multiindex* list={
mark list*\nextlist
Dark2-6\nextlist},
every axis plot/.append style={ultra thick},
xlabel={degrees of freedom},
ylabel={estimator $\eta_{\cT}$},
grid=major,
legend entries={\tiny $p=1$,\tiny $p=2$,\tiny $p=3$},
legend pos=south west,
]

\addplot table [x=ndof,y=errEst,col sep = comma] {data/Example1_orderP1_maxDof500000_theta1.csv};
\addplot table [x=ndof,y=errEst,col sep = comma] {data/Example1_orderP2_maxDof500000_theta1.csv};
\addplot table [x=ndof,y=errEst,col sep = comma] {data/Example1_orderP3_maxDof500000_theta1.csv};

\addplot [black,dotted,mark=none] table [x=ndof,y expr={3*sqrt(\thisrowno{0})^(-1)},col sep = comma] {data/Example1_orderP1_maxDof500000_theta1.csv};
\addplot [black,dotted,mark=none] table [x=ndof,y expr={3*sqrt(\thisrowno{0})^(-2)},col sep = comma] {data/Example1_orderP2_maxDof500000_theta1.csv};
\addplot [black,dotted,mark=none] table [x=ndof,y expr={3*sqrt(\thisrowno{0})^(-3)},col sep = comma] {data/Example1_orderP3_maxDof500000_theta1.csv};
\end{loglogaxis}
\end{tikzpicture}
    \begin{tikzpicture}
\begin{loglogaxis}[
    width=0.45\textwidth,
cycle list/Dark2-6,
cycle multiindex* list={
mark list*\nextlist
Dark2-6\nextlist},
every axis plot/.append style={ultra thick},
xlabel={degrees of freedom},
ylabel={estimator $\eta_{\cT}$},
grid=major,
legend entries={\tiny $p=1$,\tiny $p=1$ adap.},
legend pos=south west,
ymin=5e-3,ymax=20,
]

\addplot table [x=ndof,y=errEst,col sep = comma] {data/Example2_orderP1_maxDof500000_theta1.csv};
\addplot table [x=ndof,y=errEst,col sep = comma] {data/Example2_orderP1_maxDof1000000_theta0.25.csv};

\end{loglogaxis}
\end{tikzpicture}
\begin{tikzpicture}
\begin{loglogaxis}[
    width=0.45\textwidth,
cycle list/Dark2-6,
cycle multiindex* list={
mark list*\nextlist
Dark2-6\nextlist},
every axis plot/.append style={ultra thick},
xlabel={degrees of freedom},
ylabel={estimator $\eta_{\cT}$},
grid=major,
legend entries={\tiny $p=2$,\tiny $p=2$ adap.},
legend pos=south west,
ymin=5e-3,ymax=20,
]

\addplot table [x=ndof,y=errEst,col sep = comma] {data/Example2_orderP2_maxDof500000_theta1.csv};
\addplot table [x=ndof,y=errEst,col sep = comma] {data/Example2_orderP2_maxDof1000000_theta0.25.csv};

\end{loglogaxis}
\end{tikzpicture}
\begin{tikzpicture}
\begin{loglogaxis}[
    width=0.45\textwidth,
cycle list/Dark2-6,
cycle multiindex* list={
mark list*\nextlist
Dark2-6\nextlist},
every axis plot/.append style={ultra thick},
xlabel={degrees of freedom},
ylabel={estimator $\eta_{\cT}$},
grid=major,
legend entries={\tiny $p=3$,\tiny $p=3$ adap.},
legend pos=south west,
]

\addplot table [x=ndof,y=errEst,col sep = comma] {data/Example2_orderP3_maxDof500000_theta1.csv};
\addplot table [x=ndof,y=errEst,col sep = comma] {data/Example2_orderP3_maxDof1000000_theta0.25.csv};

\end{loglogaxis}
\end{tikzpicture}
  \end{center}
  \caption{Left: Smooth solution and $d=1$. The dotted black lines correspond to $\OO(\dim(V_{\cT,p})^{-p/2})$. Right: Pulse for $d=1$ with uniform and adaptive refinements.}\label{fig:1dsmooth}
\end{figure}

\subsection{Gaussian pulse for $d=1$}
We consider a similar example as in~\cite[Section~6.4]{GopalakrishnanS_19} with a Gaussian pulse profile. 
In particular, we set our data to
\begin{align*}
  f &= 0 = \g, \quad
  v_0 = 2\kappa(x-\mu) e^{-\kappa(x-\mu)^2}, \quad \ssigma_0 = -v_0
\end{align*}
where $\kappa = 1000$, $\mu = 0.2$.
We note that an exact solution is not known. We compare results on a sequence of uniformly refined meshes and a sequence of locally refined meshes for $p=1,2,3$ where we use the bulk criterion for marking elements for refinement. 
Figure~\ref{fig:1dsmooth} visualizes the results. For $p=1$ it seems that, for both uniformly and adaptively refined meshes, there is a long preasymptotic phase. For $p=2$ and $p=3$ convergence is much better, particularly, when using the adaptive algorithm, though it is hard to identify experimental orders of convergence.

\subsection{Solution with non-matching boundary condition for $d=1$}\label{sec:num:1d:ex3}
We consider the data $f = 0=\g=\ssigma_0$, $v_0 = 1$. The exact solution $(v,\ssigma)$ is piecewise constant on the triangles 
\begin{align*}
  T_1 = \triangle(z_1,z_2,z_5), \, 
  T_2 = \triangle(z_2,z_3,z_5), \, 
  T_3 = \triangle(z_3,z_4,z_5), \, 
  T_4 = \triangle(z_4,z_1,z_5), \, 
\end{align*}
where $z_1 = (0,0)$, $z_2 = (1,0)$, $z_3(1,1)$, $z_4 = (0,1)$, and $z_5 = (\tfrac12,\tfrac12)$. To be more precise, 
\begin{align*}
  v(t,x) = \begin{cases}
    -1 & (t,x) \in T_2, \\
    1 &(t,x) \in T_4, \\
    0 & \text{else},
  \end{cases} \qquad
  \ssigma(t,x) = \begin{cases}
    1 & (t,x) \in T_1, \\
    -1 &(t,x) \in T_3, \\
    0 & \text{else}.
  \end{cases}
\end{align*}
Due to jumps in the solution, one expects reduced convergence. This is observed in Figure~\ref{fig:1djumps} for sequence of uniformly refined meshes. The adaptive loop improves convergence after a pre-asymptotic phase but the rates are not optimal with respect to the order $p$.
\begin{figure}
  \begin{center}
    \begin{tikzpicture}
\begin{loglogaxis}[
    width=0.45\textwidth,
cycle list/Dark2-6,
cycle multiindex* list={
mark list*\nextlist
Dark2-6\nextlist},
every axis plot/.append style={ultra thick},
xlabel={degrees of freedom},
grid=major,
legend entries={\tiny $p=1$,\tiny $p=1$ adap.,\tiny $\alpha=1/8$,\tiny $\alpha=1/6$},
legend pos=south west,
ymax=1,
]

\addplot table [x=ndof,y=errEst,col sep = comma] {data/Example3_orderP1_maxDof1000000_theta1.csv};
\addplot table [x=ndof,y=errEst,col sep = comma] {data/Example3_orderP1_maxDof1000000_theta0.25.csv};

\addplot [black,dotted,mark=none] table [x=ndof,y expr={8e-1*sqrt(\thisrowno{0})^(-1/4)},col sep = comma] {data/Example3_orderP1_maxDof1000000_theta1.csv};
\addplot [black,dotted,mark=none] table [x=ndof,y expr={2e0*sqrt(\thisrowno{0})^(-1/3)},col sep = comma] {data/Example3_orderP1_maxDof1000000_theta0.25.csv};
\end{loglogaxis}
\end{tikzpicture}
\begin{tikzpicture}
\begin{loglogaxis}[
    width=0.45\textwidth,
cycle list/Dark2-6,
cycle multiindex* list={
mark list*\nextlist
Dark2-6\nextlist},
every axis plot/.append style={ultra thick},
xlabel={degrees of freedom},
grid=major,
legend entries={\tiny $p=2$,\tiny $p=2$ adap.,\tiny $\alpha=1/6$,\tiny $\alpha=1/4$},
legend pos=south west,
]

\addplot table [x=ndof,y=errEst,col sep = comma] {data/Example3_orderP2_maxDof1000000_theta1.csv};
\addplot table [x=ndof,y=errEst,col sep = comma] {data/Example3_orderP2_maxDof1000000_theta0.25.csv};

\addplot [black,dotted,mark=none] table [x=ndof,y expr={1e0*sqrt(\thisrowno{0})^(-1/3)},col sep = comma] {data/Example3_orderP2_maxDof1000000_theta1.csv};
\addplot [black,dotted,mark=none] table [x=ndof,y expr={2e0*sqrt(\thisrowno{0})^(-1/2)},col sep = comma] {data/Example3_orderP2_maxDof1000000_theta0.25.csv};
\end{loglogaxis}
\end{tikzpicture}
\begin{tikzpicture}
\begin{loglogaxis}[
    width=0.45\textwidth,
cycle list/Dark2-6,
cycle multiindex* list={
mark list*\nextlist
Dark2-6\nextlist},
every axis plot/.append style={ultra thick},
xlabel={degrees of freedom},
grid=major,
legend entries={\tiny $p=3$,\tiny $p=3$ adap.,\tiny $\alpha=1/6$,\tiny $\alpha=5/16$},
legend pos=south west,
]

\addplot table [x=ndof,y=errEst,col sep = comma] {data/Example3_orderP3_maxDof500000_theta1.csv};
\addplot table [x=ndof,y=errEst,col sep = comma] {data/Example3_orderP3_maxDof1000000_theta0.25.csv};

\addplot [black,dotted,mark=none] table [x=ndof,y expr={1e0*sqrt(\thisrowno{0})^(-1/3)},col sep = comma] {data/Example3_orderP3_maxDof500000_theta1.csv};
\addplot [black,dotted,mark=none] table [x=ndof,y expr={3e0*sqrt(\thisrowno{0})^(-5/8)},col sep = comma] {data/Example3_orderP3_maxDof1000000_theta0.25.csv};
\end{loglogaxis}
\end{tikzpicture}\includegraphics[width=0.45\textwidth,height=5cm]{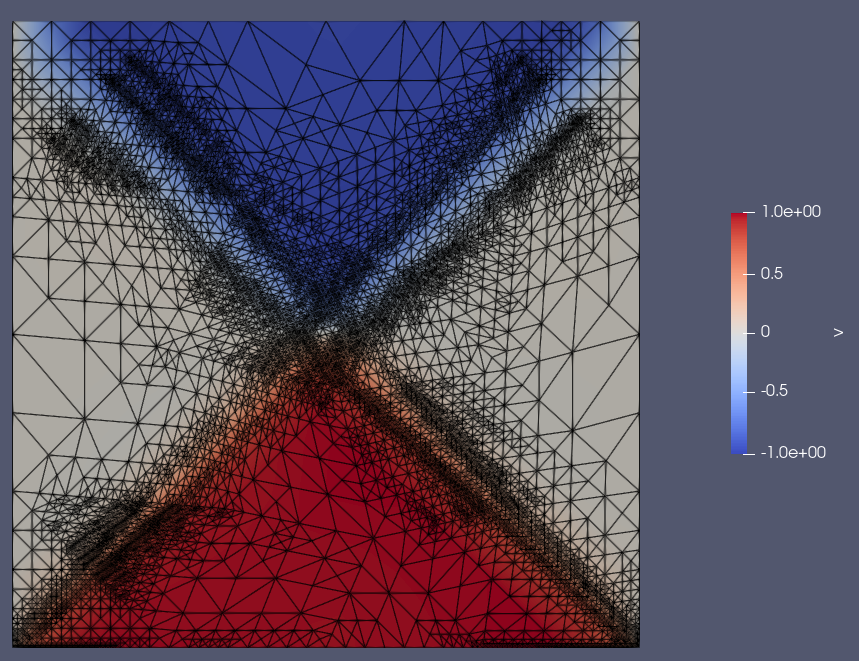}
  \end{center}
  \caption{Estimator for sequence of uniform and locally refined meshes for the example from Section~\ref{sec:num:1d:ex3}. 
    Dashed black lines indicate $\OO(\dim(V_{\cT,p})^{-\alpha})$.
  The bottom right plot visualizes the solution on a mesh with 56160 elements for $p=1$. Horizontal axis corresponds to spatial domain $\Omega = (0,1)$.}\label{fig:1djumps}
\end{figure}
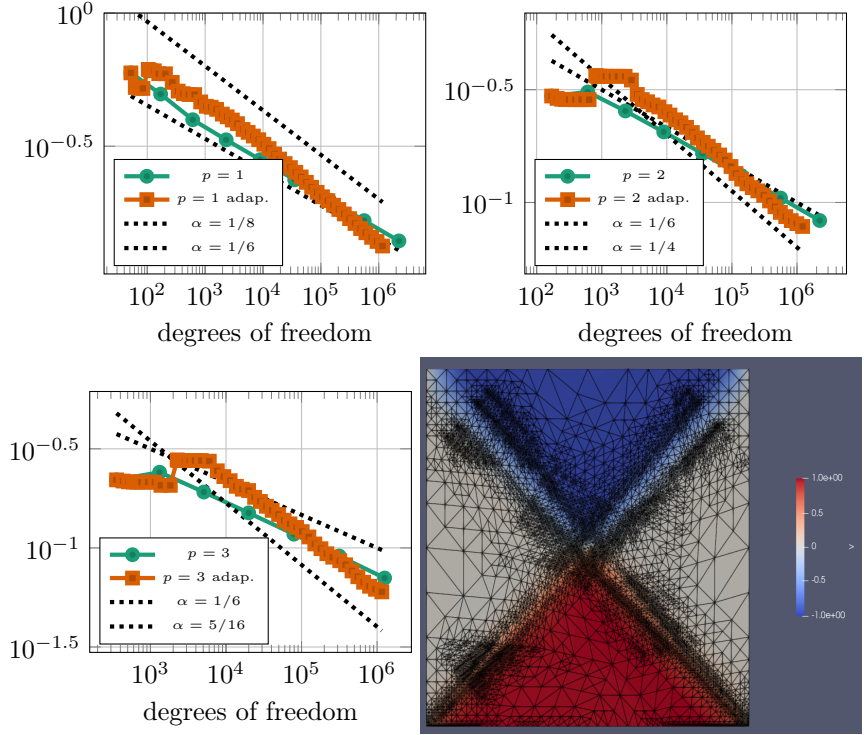
In Figure~\ref{fig:1djumps} we also see the visualization of the solution component $v$ on locally refined mesh. As expected, we observe strong refinements towards the characteristics $t = x$, $t = 1-x$ and at the corner vertices where $t =0$.

\subsection{Smooth solution for $d=2$}
We consider $u(t,x) = \tfrac12 t^2 \sin(\pi x)\sin(\pi y)$. Then, $v = \partial_t u$, $\ssigma = \partial_x u$ solve the first-order system wave equation with data $f= \partial_{tt}u - \Deltasp u$, $\g = v_0=\ssigma_0 = 0$.
On a sequence of uniform refinements of an initial triangulation of $Q$, Corollary~\ref{cor:rates} predicts rates which are perfectly aligned with the obtained results, see Figure~\ref{fig:2dsmooth}.

\begin{figure}
  \begin{center}
    \begin{tikzpicture}
\begin{loglogaxis}[
    width=0.45\textwidth,
cycle list/Dark2-6,
cycle multiindex* list={
mark list*\nextlist
Dark2-6\nextlist},
every axis plot/.append style={ultra thick},
xlabel={degrees of freedom},
ylabel={estimator $\eta_{\cT}$},
grid=major,
legend entries={\tiny $p=1$,\tiny $p=2$,\tiny $p=3$},
legend pos=south west,
]

\addplot table [x=ndof,y=errEst,col sep = comma] {data/Example12D_orderP1_maxDof1000000_theta1.csv};
\addplot table [x=ndof,y=errEst,col sep = comma] {data/Example12D_orderP2_maxDof1000000_theta1.csv};
\addplot table [x=ndof,y=errEst,col sep = comma] {data/Example12D_orderP3_maxDof1000000_theta1.csv};

\addplot [black,dotted,mark=none] table [x=ndof,y expr={3*(\thisrowno{0})^(-1/3)},col sep = comma] {data/Example12D_orderP1_maxDof1000000_theta1.csv};
\addplot [black,dotted,mark=none] table [x=ndof,y expr={2e1*(\thisrowno{0})^(-2/3)},col sep = comma] {data/Example12D_orderP2_maxDof1000000_theta1.csv};
\addplot [black,dotted,mark=none] table [x=ndof,y expr={1e2*(\thisrowno{0})^(-3/3)},col sep = comma] {data/Example12D_orderP3_maxDof1000000_theta1.csv};
\end{loglogaxis}
\end{tikzpicture}
    \begin{tikzpicture}
\begin{loglogaxis}[
    width=0.45\textwidth,
cycle list/Dark2-6,
cycle multiindex* list={
mark list*\nextlist
Dark2-6\nextlist},
every axis plot/.append style={ultra thick},
xlabel={degrees of freedom},
ylabel={estimator $\eta_{\cT}$},
grid=major,
legend entries={\tiny $p=1$,\tiny $p=1$ adap.,\tiny $\alpha=7/27$,\tiny $\alpha=1/3$},
legend pos=south west,
]

\addplot table [x=ndof,y=errEst,col sep = comma] {data/Example22D_orderP1_maxDof1000000_theta1.csv};
\addplot table [x=ndof,y=errEst,col sep = comma] {data/Example22D_orderP1_maxDof1000000_theta0.25.csv};

\addplot [black,dashed,mark=none] table [x=ndof,y expr={2e1*(\thisrowno{0})^(-1/3*7/9)},col sep = comma] {data/Example22D_orderP1_maxDof1000000_theta1.csv};
\addplot [black,dotted,mark=none] table [x=ndof,y expr={1e1*(\thisrowno{0})^(-1/3)},col sep = comma] {data/Example22D_orderP1_maxDof1000000_theta0.25.csv};
\end{loglogaxis}
\end{tikzpicture}
\begin{tikzpicture}
\begin{loglogaxis}[
    width=0.45\textwidth,
cycle list/Dark2-6,
cycle multiindex* list={
mark list*\nextlist
Dark2-6\nextlist},
every axis plot/.append style={ultra thick},
xlabel={degrees of freedom},
ylabel={estimator $\eta_{\cT}$},
grid=major,
legend entries={\tiny $p=2$,\tiny $p=2$ adap.,\tiny $\alpha=1/2$,\tiny $\alpha=2/3$},
legend pos=south west,
]

\addplot table [x=ndof,y=errEst,col sep = comma] {data/Example22D_orderP2_maxDof1000000_theta1.csv};
\addplot table [x=ndof,y=errEst,col sep = comma] {data/Example22D_orderP2_maxDof1000000_theta0.25.csv};

\addplot [black,dashed,mark=none] table [x=ndof,y expr={3e2*(\thisrowno{0})^(-1.5/3)},col sep = comma] {data/Example22D_orderP2_maxDof1000000_theta1.csv};
\addplot [black,dotted,mark=none] table [x=ndof,y expr={1e2*(\thisrowno{0})^(-2/3)},col sep = comma] {data/Example22D_orderP2_maxDof1000000_theta0.25.csv};
\end{loglogaxis}
\end{tikzpicture}
\begin{tikzpicture}
\begin{loglogaxis}[
    width=0.45\textwidth,
cycle list/Dark2-6,
cycle multiindex* list={
mark list*\nextlist
Dark2-6\nextlist},
every axis plot/.append style={ultra thick},
xlabel={degrees of freedom},
ylabel={estimator $\eta_{\cT}$},
grid=major,
legend entries={\tiny $p=3$,\tiny $p=3$ adap.,\tiny $\alpha=2/3$,\tiny $\alpha=1$},
legend pos=south west,
]

\addplot table [x=ndof,y=errEst,col sep = comma] {data/Example22D_orderP3_maxDof1000000_theta1.csv};
\addplot table [x=ndof,y=errEst,col sep = comma] {data/Example22D_orderP3_maxDof1000000_theta0.25.csv};

\addplot [black,dashed,mark=none] table [x=ndof,y expr={3e3*(\thisrowno{0})^(-2/3)},col sep = comma] {data/Example22D_orderP3_maxDof1000000_theta1.csv};
\addplot [black,dotted,mark=none] table [x=ndof,y expr={4e3*(\thisrowno{0})^(-3/3)},col sep = comma] {data/Example22D_orderP3_maxDof1000000_theta0.25.csv};
\end{loglogaxis}
\end{tikzpicture}
  \end{center}
  \caption{Top Left: Smooth solution and $d=2$. The dotted black lines correspond to $\OO(\dim(V_{\cT,p})^{-p/3})$. 
  Other: Pulse for $d=2$ with uniform and adaptive refinements. The dotted black lines correspond to $\OO(\dim(V_{\cT,p})^{-\alpha})$.}\label{fig:2dsmooth}
\end{figure}
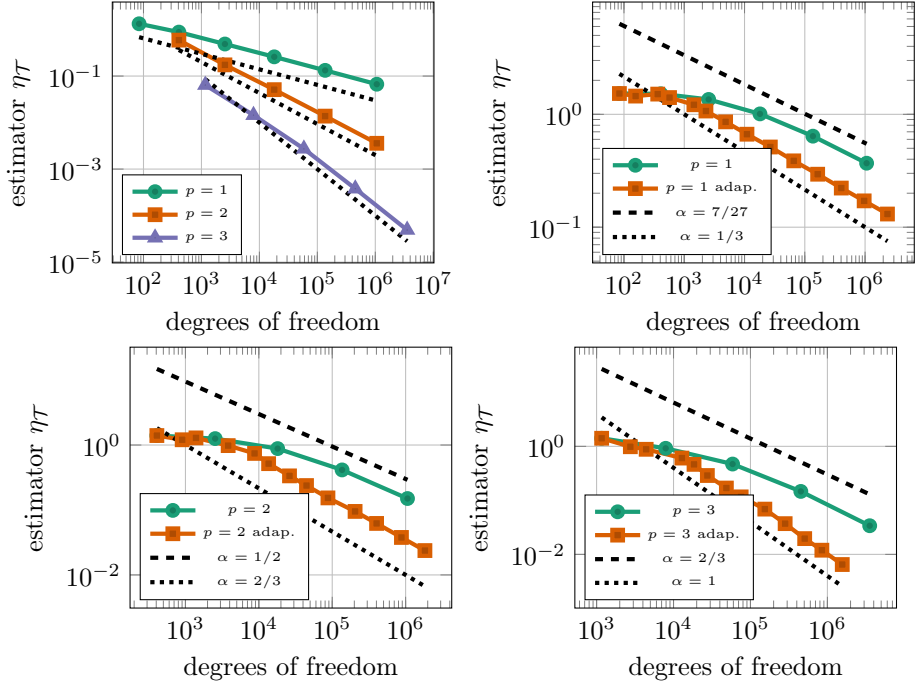

\subsection{Gaussian pulse for $d=2$}
We consider a pulse defined similar to~\cite[Section~6.6]{GopalakrishnanS_19} where we set
\begin{align*}
  v_0(x,y) &= -e^{-200\big( (x-0.2)^2 + (y-0.2)^2\big)}, \quad
  \ssigma_0(x,y) = -v_0(x,y)\begin{pmatrix} 1\\ 1\end{pmatrix}, \\
  f(t,x,y) &= 0,  \quad
  \g(t,x,y) = 400e^{-200\big( (x-0.2-t)^2 + (y-0.2-t)^2\big)} \begin{pmatrix} y-0.2-t\\x-0.2-t\end{pmatrix}.
\end{align*}
Figure~\ref{fig:2dsmooth} (right plot) displays the error estimator for sequences of uniformly and adaptively refined meshes. We find that adaptivity helps to achieve better rates in comparison. Though this might be a preasymptotic phenomena. 

\subsection{Solution with non-matching boundary condition for $d=2$}\label{sec:num:2d:ex3}
Here we set $f=\g=\ssigma_0=0$ and $v_0=1$. Note that $v_0$ does not match the boundary condition, i.e., $v_0\notin H_0^1(\Omega)$. 

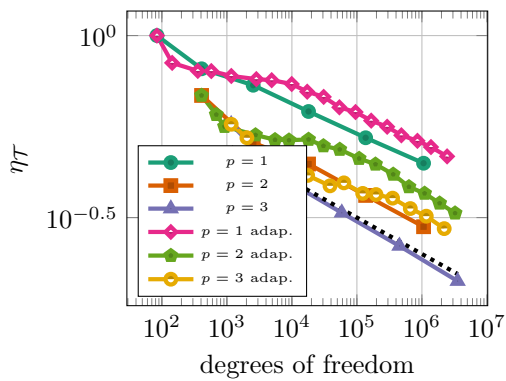
\begin{figure}
  \begin{center}
    \begin{tikzpicture}
\begin{loglogaxis}[
    width=0.49\textwidth,
cycle list/Dark2-6,
cycle multiindex* list={
mark list*\nextlist
Dark2-6\nextlist},
every axis plot/.append style={ultra thick},
xlabel={degrees of freedom},
ylabel={$\eta_{\cT}$},
grid=major,
legend entries={\tiny $p=1$,\tiny $p=2$,\tiny $p=3$,\tiny $p=1$ adap.,\tiny $p=2$ adap.,\tiny $p=3$ adap.},
legend pos=south west,
]

\addplot table [x=ndof,y=errEst,col sep = comma] {data/Example32D_orderP1_maxDof1000000_theta1.csv};
\addplot table [x=ndof,y=errEst,col sep = comma] {data/Example32D_orderP2_maxDof1000000_theta1.csv};
\addplot table [x=ndof,y=errEst,col sep = comma] {data/Example32D_orderP3_maxDof1000000_theta1.csv};
\addplot table [x=ndof,y=errEst,col sep = comma] {data/Example32D_orderP1_maxDof2000000_theta0.25.csv};
\addplot table [x=ndof,y=errEst,col sep = comma] {data/Example32D_orderP2_maxDof2000000_theta0.25.csv};
\addplot table [x=ndof,y=errEst,col sep = comma] {data/Example32D_orderP3_maxDof2000000_theta0.25.csv};

\addplot [black,dotted,mark=none] table [x=ndof,y expr={(\thisrowno{0})^(-0.3/3)},col sep = comma] {data/Example32D_orderP3_maxDof1000000_theta1.csv};
\end{loglogaxis}
\end{tikzpicture}
  \end{center}
  \caption{Estimator for sequence of uniform and locally refined meshes for the example from Section~\ref{sec:num:2d:ex3}.
  The black dashed line indicates $\OO(\#\mathrm{dof}^{-0.1})$.}\label{fig:2djumps}
\end{figure}

Figure~\ref{fig:2djumps} visualizes the error estimators on sequences of uniformly and adaptively refined meshes. We observe the reduced convergence $\OO(\#\mathrm{dof}^{-0.1})$ which is not even improved by the adaptive algorithm.
The reason of the reduced convergence is due to the non-matching boundary conditions and choice of the approximation space. Similar effects have been observed for the least-squares FEM for parabolic problems, see~\cite{FuehrerK_21}. 
For parabolic problems convergence can be improved by introducing new finite elements as has been shown in~\cite{GantnerS_22improved}.
In the future, we plan to study construction of novel finite element spaces.

\bibliographystyle{abbrv}
\bibliography{literature}
\end{document}